\documentclass[12pt, reqno]{amsart}
\usepackage{amsmath,amssymb,amsfonts,amscd,hyperref,color}
\usepackage[utf8]{inputenc}
\usepackage{csquotes}
\usepackage{nicefrac}

\usepackage[abs]{overpic}
\usepackage{verbatim}
\usepackage{enumerate}

\usepackage[
backend=biber,
style=alphabetic,
sorting=nyt,
maxnames = 99,
minnames = 99,
maxalphanames=99,
maxcitenames=99,
maxbibnames = 99,
doi=false,
url=false,
isbn= false
]{biblatex}

\AtEveryBibitem{%
	\clearfield{issn}%
}

\usepackage{xcolor}
\usepackage{xparse}
\usepackage{tikz}
\usetikzlibrary{fadings}

\definecolor{rainbowred}{HTML}{D00000}
\definecolor{rainboworange}{HTML}{E86A00}
\definecolor{rainbowyellow}{HTML}{B88600}
\definecolor{rainbowgreen}{HTML}{007A32}
\definecolor{rainbowblue}{HTML}{0047AB}
\definecolor{rainbowviolet}{HTML}{6A0D83}

\pgfdeclarehorizontalshading{rainbow shading}{100bp}{%
  color(0bp)=(rainbowred);%
  color(33.30bp)=(rainbowred);%
  color(33.31bp)=(rainboworange);%
  color(41.65bp)=(rainboworange);%
  color(41.66bp)=(rainbowyellow);%
  color(50.00bp)=(rainbowyellow);%
  color(50.01bp)=(rainbowgreen);%
  color(58.34bp)=(rainbowgreen);%
  color(58.35bp)=(rainbowblue);%
  color(66.68bp)=(rainbowblue);%
  color(66.69bp)=(rainbowviolet);%
  color(100bp)=(rainbowviolet)%
}

\newsavebox{\rainbowbox}
\newcounter{rainbowtextcounter}

\newlength{\rainbowtextaboveskip}
\newlength{\rainbowtextbelowskip}
\newlength{\rainbowparagraphindent}

\makeatletter

\newcommand{\rainbowNoHyperBegin}{%
  \@ifpackageloaded{hyperref}{\NoHyper}{}%
}

\newcommand{\rainbowNoHyperEnd}{%
  \@ifpackageloaded{hyperref}{\endNoHyper}{}%
}

\makeatother

\NewDocumentEnvironment{rainbowtext}{O{\linewidth}+b}
{%
  \par

  \vskip\rainbowtextaboveskip

  \begingroup

  \stepcounter{rainbowtextcounter}%
  \edef\rainbowfadingname{%
    rainbow-fading-\arabic{rainbowtextcounter}%
  }%

  \begin{lrbox}{\rainbowbox}%
    \rainbowNoHyperBegin
    \begin{minipage}{#1}%
      \color{white}%

      \setlength{\parskip}{0pt}%

      \setlength{\parindent}{\rainbowparagraphindent}%

      \everypar{%
        \hspace*{\rainbowparagraphindent}%
      }%

      \ignorespaces
      #2%
    \end{minipage}%
    \rainbowNoHyperEnd
  \end{lrbox}%

  \begin{tikzfadingfrompicture}[name=\rainbowfadingname]
    \node[
      inner sep=0pt,
      outer sep=0pt
    ] {%
      \usebox{\rainbowbox}%
    };
  \end{tikzfadingfrompicture}%

  \noindent
  \begin{tikzpicture}[baseline=(rainbow-node.base)]

    \node[
      inner sep=0pt,
      outer sep=0pt,
      opacity=0
    ] (rainbow-node) {%
      \usebox{\rainbowbox}%
    };

    \shade[
      path fading=\rainbowfadingname,
      fit fading=false,
      shading=rainbow shading
    ]
      (rainbow-node.south west)
      rectangle
      (rainbow-node.north east);

  \end{tikzpicture}%

  \endgroup

  \par

  \vskip\rainbowtextbelowskip

  \ignorespacesafterend
}
{}

\newcommand{\Hmm}[1]{\leavevmode{\marginpar{\tiny%
$\hbox to 0mm{\hspace*{-0.5mm}$\leftarrow$\hss}%
\vcenter{\vrule depth 0.1mm height 0.1mm width \the\marginparwidth}%
\hbox to
0mm{\hss$\rightarrow$\hspace*{-0.5mm}}$\\\relax\raggedright #1}}}

\newtheorem{theorem}{Theorem}
\newtheorem{corollary}[theorem]{Corollary}
\newtheorem{lemma}[theorem]{Lemma}
\newtheorem{proposition}[theorem]{Proposition}

\theoremstyle{definition}

\newtheorem{example}[theorem]{Example}
\newtheorem*{remark}{Remark}

\numberwithin{equation}{section}
\newcommand{\Z}{{\mathbb Z}}
\newcommand{\R}{{\mathbb R}}

\newcommand{\N}{{\mathbb N}}

\newcommand{\D}{{\mathbb D}}

\let\H\undefined
\let\L\undefined
\newcommand{\H}{H}
\newcommand{\L}{\mathcal{L}}
\newcommand{\G}{\mathcal{G}}
\newcommand{\F}{\mathcal{F}}
\renewcommand{\P}{\mathcal{P}}

\renewcommand{\L}{\mathcal{L}}
\renewcommand{\D}{\mathcal{D}}

\newcommand{\eat}[1]{}

\title[Liouville Theorem for Domains in Graphs]
{A Liouville Theorem for Domains in Graphs}

\author[P.~Hake]{Philipp Hake}
\address{P.~Hake, Institut für Mathematik, Universität Leipzig, 04109 Leipzig, Germany}\email{philipp.hake@math.uni-leipzig.de}
\author[M.~Keller]{Matthias Keller}
\address{M.~Keller, Israel Institute of Advanced Studies, Jerusalem, Israel; Institut für Mathematik, Universität Potsdam
14476  Potsdam, Germany}
\email{matthias.keller@uni-potsdam.de}
\author[F.~Pogorzelski]{Felix Pogorzelski}
\address{ F.~Pogorzelski, Israel Institute of Advanced Studies, Jerusalem, Israel;  Institut für Mathematik, Universität Leipzig, 04109 Leipzig, Germany}
\email{felix.pogorzelski@math.uni-leipzig.de}

\author[M.~Schmidt]{Marcel Schmidt}
\address{ M.~Schmidt,  Institut für Mathematik, Friedrich-Schiller-Universität Jena, 07743 Jena, Germany}
\email{schmidt.marcel@uni-jena.de}

\begin{document}

\date{\today}

\begin{abstract}
We show a Liouville theorem for domains  in graphs with Dirichlet boundary conditions. More specifically, we characterize the non-existence of non-zero bounded harmonic functions. Since Dirichlet boundary conditions give rise to Laplacians with a positive killing term, we can first characterize the validity of the Liouville theorem by the fact that the Green operator applied to the killing terms is equal to 1, or in other words, that the constant function $1$ is a potential. Secondly, we derive a characterization in terms of stochastic completeness at infinity and and total loss of heat. Thirdly, we give a characterization in terms of a Green formula for superharmonic potentials. Finally, we investigate the Liouville property in terms of recurrence and transience of the graph without killing term. As an application, we consider subsets of the Euclidean space such as cones and percolation clusters, and weakly spherically symmetric graphs.
\end{abstract}
\maketitle

\tableofcontents

\section{Introduction and main result}\label{sec:intro}
The classical Liouville theorem says that on $\R^d$ there are no non-constant bounded harmonic functions, \cite{Liouville,Bocher,Picard1923}. Since then variations of such a result have been proven in various contexts. For example Yau's result on complete Riemannian manifolds \cite{Yau76} later refined by Karp \cite{Karp82b} states that there are no non-constant harmonic functions in $ L^p $    or the seminal work of Colding/Minicozzi \cite{CM}  bounds the number of linearly independent bounded harmonic functions of polynomial growth.

We revisit this problem on graphs where the classical Liouville theorem on $\Z^d$ is well-known \cite{Heilbronn,Capoulade1932}. The uniqueness of positive harmonic functions on orthonants of $\Z^d$ was studied in \cite{BMS,Raschel} and a Phragmen-Lindelöf type result  was shown in \cite{Guadie}. Recently, a rather surprising result by Buhovsky/Logunov/Malinnikova/Sodin showed that on $\Z^2$ there are no non-constant harmonic functions which are unbounded on a small set, \cite{BLMS}, which is in stark contrast to the case of $\Z^3$ for $d \geq 3$. Furthermore, there are also so called three-circle theorems for harmonic functions on $\Z^d$,  \cite{GM,LM1,LM2}, which bound the average growth of harmonic functions.

Liouville theorems have also been studied intensively on groups, see e.g.\@ the influential works \cite{KV83,Erschler,Kaimanovich96}, and also the classical textbook by Woess \cite{Woe00} for abstract background and the connection to the Martin boundary. For subgraphs induced by percolation clusters or random environments in Cayley graphs, Liouville theorems have been obtained in  \cite{Kaimanovich90,BLS99,Barlow,Gab05,BDKY15}.
There are also results on the dimension of the space of polynomially bounded  harmonic functions on groups in the spirit of Colding/Minicozzi, see e.g.\@ \cite{Kleiner,HJL} which were also generalized to graphs satisfying certain curvature bounds \cite{HJ,HJL2}. As for results on minimal growth conditions on non-constant harmonic functions on groups, see \cite{BDKY15,BDKY17,AK24}.
 
 For general graphs satisfying a Harnack inequality, the  Liouville theorem for positive harmonic functions was
 proved by Delmotte \cite{Delmotte}. It was then subsequently used in various works, see e.g.\@ \cite{Barlow,Hua19,Hua21,HLLY,Munch}. Also, analogues of Yau's and Karp's result have been proven for graphs \cite{HuaJost13, HuaKeller13,KellerLenzSchmidt2026} and even more generally for Dirichlet forms \cite{HuaKellerLenzSchmidt}. The connection of $\ell^2$-Liouville theorems and essential self-adjointness was recently elaborated on in \cite{HMW21,IKMW}. In \cite{AS23}, it has been shown that every stochastically complete graph over a discrete measure space has the $\ell^1$-Liouville property, i.e.\@ every positive summable superharmonic function is constant. The Liouville property has been verified in \cite{BMP24} for bounded solutions to the discrete homogeneous  Schr\"odinger equation  with positive potentials satisfying a minimal growth condition with respect to an intrinsic metric. A related result for certain $\ell^p$-solutions can be found in \cite{MP25}. 

Liouville properties can be quite sensitive with respect to quasi-isometry of graphs. While stability holds for harmonic functions of finite Dirichlet energy 
\cite{Soa93}, instability for bounded harmonic functions was established in \cite{Benjamini}.

The existence of non-constant bounded harmonic functions on transient planar graphs was established in \cite{BenjaminiSchramm,Hutchcroft,LPS}.

The question we study in this paper is  the non-existence of non-zero bounded harmonic functions on domains, i.e., subgraphs of graphs with Dirichlet boundary conditions. Notice that   a Dirichlet boundary condition of a subgraph translates into the presence of a positive killing term for the corresponding Laplace operator, see below. Therefore, it is obvious that under the presence of such a boundary the constant functions are no longer harmonic unless they are constantly zero. Moreover, on a finite subgraph,  the corresponding Laplace operator with Dirichlet boundary conditions is injective \cite[Lemma~0.29]{KLW} and, therefore,  no non-zero harmonic functions can exist. However, as we will demonstrate, on infinite subgraphs, non-zero bounded harmonic functions may occur.

Let us first introduce the setting and then state the main result of this paper.
A {\em graph} $(b,c)$ with vertices in a countable, discrete set $X$ is a pair consisting of a symmetric function $b:X \times X \to [0, \infty)$ with $b(x,x) = 0$ and $$ 
\sum_{y \in X} b(x,y) < \infty $$ for all $x \in X$, and a non-negative map $c:X \to [0,\infty)$. The map $c$ is called the {\em killing term} of the graph. If $c=0$, then we simply write $b$ for the graph $(b,0)$. Furthermore, we define the canoncial $\ell^p(X)$ spaces with norms ${\|\cdot\|}_p$  for $p \in [1,\infty]$ and scalar product $\langle{\cdot}, \cdot\rangle$ for $ p=2 $ in the usual way. We call a function $f$ which takes values in $[0,\infty)$ (respectively in $(0,\infty	)$) \emph{positive} and write $f\ge0$ (respectively \emph{strictly positive} and write $f>0$), and call $-f$ \emph{negative} (respectively \emph{strictly negative}).
For a summable or positive function $f$ over a discrete set $A$ (think of  a subset of $X$ or $X\times X$), we write $$\sum_{X} f=\sum_{x \in X} f(x).$$
Given $x \in X$, we write $1_x = 1_{\{x\}} \in C_c(X)$.

We will assume throughout this work that the graphs $(b,c)$  are {\em connected}, i.e.,\@ for every  $x,y \in X$ there are $x= y_1,y_2, \dots, y_n, y_{n+1}=y$ such that $b(y_i, y_{i+1})>0$ for all $1 \leq i \leq n$.

 The {\em formal Laplacian} of a  graph $(b,c)$ over $X$ is given by a map  $\L=\L_{b,c}:\mathcal{F} \to C(X)$ defined on $$\mathcal{F}=\{f:X\to \R\mid \sum_{y\in X} b(x,y)|f(y)|<\infty\mbox{ for all }x\in X\},$$
where we denote by $C(X)$ the vector space of real-valued functions on $X$ which acts as
\[\L f(x)= \sum_{y\in X}b(x,y)(f(x)-f(y)) \,+ \, {c(x)} f(x).\]

There is a related	quadratic form $\mathcal{Q}$  defined 
 as 
\[\mathcal{Q}(f)=\frac{1}{2}\sum_{x,y\in X}b(x,y)(f(x)-f(y))^2+\sum_{x\in X}c(x)f(x)^2\] 
on the space of {\em functions of finite energy} \[\mathcal{D}=\mathcal{D}(\mathcal{Q})=\{f\colon X\to \R \mid \mathcal{Q}(f)<\infty\}.\]
We denote by $C_c(X) \subseteq C(X)$  the subspace of all finitely supported functions. 
The {\em extended space} $$\mathcal{D}_0=\mathcal{D}_0(\mathcal{Q})=\overline{C_c(X)}^{\|\cdot\|_{o}}$$ is defined as the closure of $C_c(X)$ in $\mathcal{D}$ with respect to the norm $\|f\|_{o}=\sqrt{\mathcal{Q}(f)+|f(o)|^2}$, where $o \in X$ is some fixed vertex. The restriction $Q=Q_{b,c}$ of $\mathcal{Q}$ to $$D(Q)=\mathcal{D}_0\cap \ell^2(X)$$ gives rise to a regular Dirichlet form, see \cite[Theorem~1.19]{KLW}, and we denote the associated self-adjoint operator on $\ell^2(X)$  by $L$ and observe that by \cite[Theorem~1.6]{KLW}
\begin{align*}
	L=\mathcal{L}\quad\mbox{on }D(L).
\end{align*}

Next, we explain how the killing term relates to Dirichlet boundary conditions for subgraphs which is discussed in greater detail in \cite[Section~2.2]{KLW}. Let  $b$ be a connected graph over $X$ and let $U$ be a non-empty proper subset of $X$. We let 
$ i_U \colon C(U) \to C(X)$ be the extension by zero of functions on $U$ to $X$ and define $$Q_U(f)=Q(i_Uf)$$ as a
form with domain 
\begin{align*}
D(Q_U) = \{f \in\ell^2(U) \mid  i_U f \in D(Q)\}.
\end{align*}
It was shown that this form is a Dirichlet form, \cite[Proposition~2.18]{KLW} and equal to the restriction of $\mathcal{Q}$ to the form closure of $C_c(U)$ in $\ell^2(U)$, \cite[Corollary 2.19]{KLW} which justifies to refer to $Q_U$ as the restriction of $Q$ to $U$ with \emph{Dirichlet boundary conditions}. Furthermore, the corresponding operator $L_U$ is a restriction of the formal Laplacian $\mathcal{L}_U$ of the graph $(b_U,d_U)$ over $U$ defined by $b_U(x,y) = b(x,y)$ and  $$ d_U(x) = \sum_{y \in X \setminus U} b(x,y). $$
Specifically, we have
\begin{align*}
	\mathcal{L}_U f(x) =  \sum_{y\in U}b(x,y)(f(x)-f(y)) \,+ \, {d_U(x)} f(x)
\end{align*}
for $f$ in the corresponding formal domain and $ x\in U $.
Observe that while the original graph $b$ did not have a killing term, the killing term $d_U$ arises naturally for the Dirichlet restriction of the graph to $U$.

Thus, we consider graphs with killing term as they model the Dirichlet restrictions of Laplacians to subgraphs. Furthermore, it eases notation since the specific supergraph becomes irrelevant when considering the Dirichlet boundary conditions.

A connected graph $(b,c)$ is called \emph{transient} if  
for every $y \in X$,  there is a unique $G_y =G_y^{b,c}\in \D_0$ such that $\mathcal{L} G_y = 1_y$. It is well known that whenever $c\neq 0$, then the graph is necessarily  transient. 
The function $$G=G^{b,c}\colon X \times X \to (0,\infty), \quad G(x,y) = G_y(x)$$ is called the {\em Green function} of the graph $(b,c)$. We recall that the Green function can be represented as the Laplace transform of the semigroup, \cite[Theorem~6.29]{KLW}
\begin{align*}
	G(x,y) = \int_0^\infty e^{-tL}1_y(x) dt.
\end{align*}
Letting 
\begin{align*}
 	\mathcal{G} =\mathcal{G}^{b,c} =\{k\colon X\to \R \mid\sum_{y\in X}G(x,y)|k(y)|< \infty \mbox{ for all } x \in X \}, \end{align*}
we define  the corresponding  Green operator $G=G^{b,c}\colon\mathcal{G} \to C(X)$ as
	\begin{align*}
 		Gk(x)=\sum_{y\in X}G(x,y)k(y), 
\end{align*}
for $ k \in \mathcal{G}$  and   $x \in X$.
Every $u \in \mathcal{F}$ arising as $u=Gk$ with $k \in \mathcal{G}$ is called a {\em potential}. We write $$\mathcal{P}=\{u\in \mathcal{F} \mid u=Gk \text{ for some } k \in \mathcal{G}\}$$ for the set of all potentials. Observe that for any potential $u=Gk$ we have $\mathcal{L} u=k$, \cite[Proposition~4]{HKPII}. We  also define  $Gf(x)=\sum_{y} G(x,y) f(y)\in[0,\infty]$ for all positive functions $f\ge 0$. If $c=0$, denote $G^b=G^{b,0}$.
\medskip

Furthermore, we recall the concept of stochastic completeness at infinity.
We only sketch briefly what is needed to state the theorem and refer to Section~\ref{sec:Heat} or \cite[Chapter~7]{KLW} for in-depth background information. 
A connected graph $(b,c)$ is called {\em stochastically complete} if $e^{-tL}1 = 1$ for all  $t \geq 0$, where $e^{-tL}$ denotes the semigroup associated with the self-adjoint Laplacian $L$, see \cite[Chapter~1]{KLW} which extends to $\ell^{\infty}$.
Any graph with non-vanishing killing term $c \neq 0$ is necessarily {\em stochastically incomplete}, i.e.\@ not stochastically complete, and we call $(b,c)$  {\em stochastically complete at infinity} if for some (equivalently all) $t >0$ and for some (equivalently all) $x \in X$, we have 
\begin{align*}
	M_t(x) := e^{-tL}1(x) \, + \, \int_0^{t} ( e^{-sL} {c})(x) \,ds=1.
\end{align*}
Stochastic completeness means that all  heat is conserved within the graph. On the other hand, a killing term constantly removes heat from the graph, so it becomes stochastically incomplete. Stochastic completeness at infinity then means that no heat is lost except for the amount removed by the killing term.
Clearly, every stochastically complete graph is also stochastically complete at infinity. 

Finally, we call a set $A \subseteq X$ {\em thick (with respect to $b$)} if for any $(b,0)$-superharmonic function $s\ge 0$ the inequality $s \geq 1$ on $A$ implies $s \geq 1$ on $X$.
	
\begin{theorem}[Liouville theorem]
		\label{thm:main}
Let $(b,c)$ be a connected transient graph over $X$.
Then, the following assertions are equivalent.
\begin{itemize}
\item[(i)] Every bounded harmonic function is zero.
\item[(ii)] $Gc = 1$, i.e., the constant function $1$ is a potential.
\item[(iii)]  The graph is stochastically complete at infinity and 
$$e^{-tL}1 \to 0,\mbox{ as } t \to \infty\;\mbox{ pointwise}.$$
\item[(iv)] For all  superharmonic $u \in \mathcal{P}$  $$\sum_X  \mathcal{L} u = \sum_X cu,$$
where both sums are over non-negative terms.
\item[(v)] For any non-thick $A \subseteq X$ and some (all)  $x \in X$
  $$\sum_{y \in X\setminus A} G^b(x,y) c(y) = \infty. $$
\end{itemize}
\end{theorem}

The equivalence of (i) and (ii) is found in the  context of energy forms, which include (extended) Dirichlet forms, in unpublished parts of the fourth named author's PhD thesis, see \cite[Theorem 4.55]{Sch17}. However,  owing to the discrete setting, the proof we give below is much simpler.
 In the literature, the convergence  $e^{-tL}1 \to 0$, as $t \to \infty$, in (iii) is known as $\ell^1$-asymptotic stability of the semigroup, because by a duality argument it is equivalent to $e^{-tL} f \to 0$, as $t \to \infty$, for all $f \in \ell^1(X)$. This asymptotic stability has been investigated in the context of Schrödinger operators  with nonnegative potentials on $\R^d$, see \cite{ABB92,Bat92}, and in the conext of perturbations of recurrent Dirichlet forms, see \cite{GO96}. Its equivalence to (i) and (ii) in the $\R^d$-setting is contained in \cite[Proposition~3.4]{ABB92}. The notion of stochastic completeness at infinity  was coined  in \cite{KL12} and since 
  $\R^d$ is stochastically complete, stochastic completeness at infinity holds for all nonnegative potentials and \cite{ABB92} simply uses the formula behind it (adapted to $\R^d$). 
  A related approach by Takeda is studied in \cite{Takeda} where the Liouville property is discussed under the scope of a notion called explosion by killing. A version of the equivalence of (i) and (v) in $\R^d$ was shown in \cite[Theorem~1.1]{GH98}, see also \cite{Murata03,Murata05,GangulyPinchover}. With this result we recover a discrete version of Batty's result from $\R^d$, \cite{Bat92}.  Independent of these considerations on asymptotic stability, the Liouville property (i) for Schrödinger operators with nonnegative potentials has been investigated in \cite{BCX08,HGP10}.

We devote  a section to each item in the above theorem, where we also expand on the corresponding criterion. For example we  characterize (i)  of subharmonic and superharmonic functions in Section~\ref{sec:Liouville}. We will also show that it is sufficient that the criteria (ii)  hold for some $x \in X$ in Section~\ref{sec:Green} and give a characterization in terms of random walks. 
Furthermore, we will discuss in Section~\ref{sec:Heat} that the condition (iii) means that  all heat is lost from the graph via the killing term, or in other words, via the Dirichlet boundary condition. We will also give a probabilistic interpretation of this condition in terms of the lifetime of the associated Markov process.
We will also show that criterion (iv) can be interpreted as a Green formula and that it can be reduced to functions $u=G_x$ for  $x \in X$ in Section~\ref{sec:GreenFormula}. Section~\ref{sec:Decomposition} is devoted to the decomposition criterion (v) and we put the pieces together for the proof of Theorem~\ref{thm:main}.
In Section~\ref{sec:Rec}, we show the Liouville property if the underlying graph without killing term is recurrent. The Section~\ref{sec:measure} is devoted to Schrödinger operators over general measure spaces, and we show that the equivalences still hold in this case.
Finally, in Section~\ref{sec:examples}, we apply our results to some examples. Specifically, we verify non-existence of non-zero, bounded harmonic functions for Dirichlet restrictions on subsets of  $\Z^d$ such as cones or percolation clusters in the supercritical regime, and weakly spherically symmetric graphs.  

\section{The Liouville property}\label{sec:Liouville}
In this section we take a closer look at the Liouville property and give the proof of the equivalence of (i) and (ii) in Theorem~\ref{thm:main}. We also give a characterization of the Liouville property in terms of bounded superharmonic functions being a potential  and positive subharmonic functions being zero.

\begin{theorem}[Liouville property]
	\label{thm:Liouville}Let $(b,c)$ be a connected transient graph over $X$.  Then, the following assertions are equivalent.
	\begin{itemize}
		\item[(i)] Every bounded harmonic function is zero.
		\item[(i.a)] Every bounded superharmonic function is a potential.
		\item[(i.b)] Every bounded  subharmonic function is negative. 
		\item[(i.c)] Every bounded positive subharmonic function is zero.
		\item[(ii.a)] $Gc(x) = 1$ for some (all) $x \in X$, i.e., 
		$1 \in \P$.
	\end{itemize}
\end{theorem}

We employ the  Riesz decomposition for superharmonic functions,  \cite[Theorem~2.4]{FischerKeller}.
\begin{proposition}[Riesz decomposition]
	Let $u$ be superharmonic and positive. Then there exists a unique decomposition $u = u_p + u_h$ into a potential $u_p$ and a harmonic function $u_h$. It satisfies $0 \leq u_p, u_h \leq u$.
\end{proposition}
\begin{proof}
	See \cite[Theorem~2.4]{FischerKeller} for the existence and uniqueness of the decomposition and the non-negativity of $u_h$. The latter implies $u_p \leq u$. The potential satisfies $\L u_p = \L u \geq 0$ and so $u_p = G \L u_p \geq 0$ which implies $u_h \leq u$.
\end{proof}

On the first glance, it is not clear that $c\in \mathcal{G}$ and, therefore, that $Gc$ is finite. 
However, one even obtains $Gc \leq 1$ by the Riesz decomposition which also gives a dichotomy for $Gc$  being equal to $1$ or strictly smaller than $1$.

\begin{lemma}
	\label{cor:Gc_dichotomy} We have $c \in \G$ and $Gc \leq 1$. More specifically, we have either $Gc = 1$ or $0=\inf_{  X} Gc  \le Gc<1$. Moreover,   $G c \leq 1  - u$ for any   subharmonic function $u\le 1$.
\end{lemma}	
\begin{proof}
Without loss of generality, we can assume $c \neq 0$, so that $(b,c)$ is transient. 
We have $\L 1 = c \geq 0$ and so by the Riesz decomposition there exist a potential $u_p$ and a harmonic function $u_h$ such that $1 = u_p + u_h$ with $0\le u_p,u_h \leq 1$. Then
	\begin{align*}
		u_p = G \L u_p = G \L 1 = Gc, 
	\end{align*}
	and so $c \in \G$ and $Gc \leq 1$. It follows that $u_h = 1 - u_p = 1 - Gc$, and since $u_h$ is harmonic and positive it must either hold that $u_h = 0$, i.e.,\@ $Gc = 1$, or that $u_h > 0$, i.e.\@ $Gc < 1$ by the Harnack inequality \cite[Theorem~4.1]{KLW}. If $Gc = 1$ then $1$ is a potential by definition, and if $1 \in \P$ then $1 = G \L 1 = Gc$.
Assume now that there exists $0 < \alpha \leq 1$ with $\alpha \leq Gc$. Then
$$(1-\alpha) G c \geq \alpha (1-Gc).$$

The left side is a nonnegative potential and the right side is a nonnegative harmonic function. Since potentials have $0$ as largest harmonic minorant \cite[Theorem~3.5]{FischerKeller}, we infer $\alpha(1-Gc) = 0$, i.e., $Gc = 1$.

For the last part, since $1 = Gc + (1 - Gc)$, the function $1 - Gc$ is the greatest harmonic minorant of $1$ and $u$ is a subharmonic minorant of $1$. Hence,  \cite[Theorem~3.4]{FischerKeller} implies $u \leq 1 - G c$.
\end{proof}

\begin{proof}[Proof of Theorem~\ref{thm:Liouville}]
(i) $\Rightarrow $ (ii.a): Since $\L (1-Gc) = 0$, assumption (i) implies that $1 - Gc = 0$ and so $1 = Gc$. The ``for some'' and ``for all'' equivalence follows from Lemma~\ref{cor:Gc_dichotomy}.\medskip

(ii.a) $\Rightarrow$ (i.b): Let $u$ be bounded and subharmonic, without loss of generality $u \leq 1$. Then $Gc \leq 1-u$ by Lemma~\ref{cor:Gc_dichotomy} and if (ii.a) holds then it implies that $u \leq 0$.\medskip

(i.b) $\Rightarrow$ (i.c): If $u$ is bounded, positive and subharmonic then it must also be negative by (i.b) and thus $u = 0$.\medskip

(i.c) $\Rightarrow $ (i.a):  Let $u$ be a bounded superharmonic function. Then, $ -u $ is a bounded subharmonic function and so is the positive part of $(-u)_+$, see \cite[Lemma~1.9]{KLW}. By (i.c) it must be zero and so $u \geq 0$. 
By the Riesz decomposition theorem $u=u_h+u_p$ with a bounded harmonic part $u_h$ which must be zero by (i.c).
Therefore, $u=u_p $ is a potential.\medskip

(i.a) $\Rightarrow $ (i):  Assuming (i.a) holds, (i) follows since $0$ is the only harmonic potential.
\end{proof}
 
\begin{remark}[Stochastic completeness at infinity as a special case of the Liouville property] In Section~\ref{sec:Heat} we will take a closer look at the relationship between stochastic completeness at infinity and the Liouville property. However, already at this point we make the observation that for a given $\alpha > 0$, the graph $(b,c + \alpha)$ satisfies the Liouville property if and only $(b,c)$ is stochastically complete at infinity:
according to \cite[Theorem~7.2~(v.b)]{KL12}, stochastic completeness at infinity is equivalent to triviality of every bounded $u$ with $(\mathcal L + \alpha) u = 0$ for some (all) $\alpha > 0$, cf.~\cite[Theorem~7.2~(v.a)]{KL12} for subsolutions. Indeed,  \cite[Theorem~7.2~(ii)]{KLW} or  \cite{MS20} for Riemannian manifolds, also show that this is equivalent    $G^{b,c + \alpha} (c + \alpha)= (L+\alpha)^{-1}(c + \alpha ) = 1$.  Hence, the known characterizations of stochastic completeness at infinity are precisely the content of Theorem~\ref{thm:Liouville} for potentials that are bounded from below by a positive constant. Remarkably, this characterization is so stable that it can be extended to certain nonlinear discrete Schrödinger operators with potentials bounded below by a positive constant, see \cite[Theorem 4.2]{SZ25}.

 The PhD thesis \cite[Theorem~4.55]{Sch17} of the fourth author, which contains Theorem~\ref{thm:Liouville} as a special case, removed the necessity of $c$ being bounded below by a positive constant. The proof there is more complicated than the one presented here and it relies on a maximum principle similar to  the approach in \cite[Chapter~7]{KL12} rather than the Riesz decomposition.
\end{remark}

\section{The Green operator perspective}\label{sec:Green}

In this section we take a closer look at (ii) in Theorem~\ref{thm:main}.  
To this end, we need a bit of preparation. For a finite subset $K \subseteq X$, we denote by $L_K$ the Laplace operator associated to the graph $(b_K, c_K)$ over $K$ where $b_K(x,y) = b(x,y)$, $c_K(x) = c(x) + d_K(x)$, with the boundary term given by
$$d_K(x) = \sum_{y \in X \setminus K} b(x,y)$$ for $x \in K$. The operator $L_K$ is injective \cite[Lemma~0.29]{KLW} and we denote the Green operator by $G_K = (L_K)^{-1}$. 

 An exhaustion of $X$ is a sequence $(K_n)$ of finite subsets of $X$ such that $K_n \subseteq K_{n+1}$ for all $n \in \N$ and $\bigcup_{n \in \N} K_n = X$.

 Furthermore, we  associate a random walk to the graph $(b,c)$ by defining the transition probabilities $p(x,y) = b(x,y)/\deg(x)$, where $\deg(x) = \sum_{y \in X} b(x,y)+ c(x)$ is the degree of $x$. We denote by $p_n(x,y)$ the $n$-step transition probabilities of the random walk, i.e.,\@ $p_0(x,y) = 1_{x=y}$ and $p_{n+1}(x,y) = \sum_{z \in X} p_n(x,z)p(z,y)$ for $n \geq 0$.
 We write $\mathbb{P}_x$ for the probability measure of the random walk starting at $x$. We denote  $$\tau  = \inf\{n \geq 0 \mid Y_n \notin X\},$$ where $Y_n$ is the position of the random walk at time $n$. Note that this is possible since $\sum_{y \in X} p(x,y) \leq 1$ for all $x \in X$ and, therefore, the random walk can leave the graph at $x$ with positive probability whenever $c(x) > 0$. For an in depth background on random walks see e.g. \cite{Woe00}. 

\begin{theorem}[Green operator perspective]
		\label{thm:Green}
	Let $(b,c)$ be a connected transient graph over $X$.  
Then 
the following assertions are equivalent.
\begin{itemize}
\item[(ii.a)] $Gc(x) = 1$ for some (all) $x \in X$, i.e., $1\in\P$.
\item[(ii.b)] For some (all) exhaustion(s) $(K_n)$ of $X$ it holds that for some (all) $x \in X$ 
$$\lim_{n \to \infty}  (G_{K_n}d_{K_n})(x) =0.$$ 
\item[(ii.c)] $\mathbb {P}_x(\tau<\infty) = 1$ for some (all) $x \in X$.
\end{itemize}
\end{theorem}
To show the theorem, we first prove the following formula which gives the equivalence of (ii.a) and (ii.b).
\begin{lemma}
	\label{lem:Green} Let $(b,c)$ be a connected transient graph over $X$.  Then, for any exhaustion $(K_n)$ of $X$ and all $x \in X$,
\begin{align*}
	Gc(x) = 1 - \lim_{n\to\infty} (G_{K_n}d_{K_n})(x).
\end{align*}
\end{lemma}
\begin{proof}
	We first observe that for finite subsets $K \subseteq X$, we have
	\begin{align*}
	\sum_{y\in K}G_K(x,y) {c_K}(y) &=L_K^{-1} {c_K}(x)=1
	\end{align*}
	since $L_K1 = c_K$ on $K$. Hence, with the projection $\pi_K\colon C(X) \to C(K)$ defined by $\pi_K f = f|_K$ for $f \in C(X)$, we have, for all $x\in X$ large enough, by \cite[Theorem~6.26~and Theorem~6.29]{KLW} and $c_K=c+d_K$ on $K$ that
	\begin{align*}
		Gc(x) &=\lim_{n\to\infty} (G_{K_n}\pi_{K_n}{c})(x) =\lim_{n\to\infty} \sum_{y\in K_n}G_{	K_n}(x,y) {c(y)}\\
		&=\lim_{n\to\infty} \sum_{y\in K_n}G_{K_n}(x,y) {c_{K_n}(y)} - \lim_{n\to\infty} \sum_{y\in K_n}G_{K_n}(x,y) {d_{K_n}(y)}\\
		&=1 - \lim_{n\to\infty} (G_{K_n}{d_{K_n}})(x).\hfill\qedhere
	\end{align*}
\end{proof}

Next, we show a formula which gives a probabilistic interpretation of $Gc$ as the probability that the random walk leaves the graph in finite time which gives the equivalence of (ii.a) and (ii.c) in Theorem~\ref{thm:Green}.
\begin{lemma}
	\label{lem:Green3} Let $(b,c)$ be a connected transient graph over $X$.  Then, for all $x \in X$,
\begin{align*}
	Gc(x) = \mathbb{P}_x(\tau < \infty).
\end{align*}
\end{lemma}	
\begin{proof}
	We first observe that $\mathbb{P}_x(Y_1\not\in X) = c(x)/\deg(x)$ for all $x \in X$. Then, we compute, for all $x \in X$,
\begin{multline*}
	\mathbb{P}_x(\tau < \infty)  = \sum_{n=0}^\infty \mathbb{P}_x(\tau =n+1) 
	 = \sum_{n=0}^\infty \sum_{y \in X} \mathbb{P}_x(Y_{n}=y ) \mathbb{P}_y(Y_1\not\in X)\\
	 =\sum_{n=0}^\infty \sum_{y \in X} p_n(x,y) \frac{c(y)}{\deg(y)} 
 = Gc(x),
\end{multline*}	
where we used	\cite[Theorem 3.4]{Schm} or \cite[Theorem 6.35]{KLW} for the last equality.
\end{proof}

\begin{proof}[Proof of Theorem~\ref{thm:Green}] 
	The independence of $x$ in (ii.a) was shown above in Lemma~\ref{cor:Gc_dichotomy}. Hence, it suffices to show the equivalence of (ii.a), (ii.b), and (ii.c) for fixed $x \in X$.

	The equivalence of (ii.a) and (ii.b)  follows from Lemma~\ref{lem:Green} and  the equivalence of (ii.a) and (ii.c) follows from Lemma~\ref{lem:Green3}.
\end{proof}

\begin{remark}
	The two lemmas above can be understood in view of the Riesz decomposition of the constant function $1=u_p+u_h$ into a potential and a harmonic function for all $x \in X$ as
	\begin{align*}
	u_p(x)= \mathbb{P}_x(\tau < \infty),\quad\mbox{ and }\quad	
	u_h(x)=\lim_{n \to \infty}  (G_{K_n}d_{K_n})(x).
	\end{align*} 
\end{remark}

\section{The stochastic completeness  perspective}\label{sec:Heat}
In this section we give the proof of the equivalence of (ii) and (iii) in Theorem~\ref{thm:main} and discuss the criterion (iii) in more detail. To this end recall that the semigroup $e^{-tL}$ extends to $\ell^\infty(X)$ and that the graph is stochastically complete if $e^{-tL}1 = 1$ for all $t \geq 0$. Since we are interested in the case where we have $c\ne0$, the  graph is necessarily stochastically incomplete. In \cite{KL12} a concept called stochastic completeness at infinity was introduced which means that for some (equivalently all) $t >0$ and for some (equivalently all) $x \in X$, we have
\begin{align*}M_t(x) = e^{-tL}1(x) \, + \, \int_0^{t} ( e^{-sL} {c})(x) \,ds=1.
\end{align*}
While $ e^{-tL}1(x) $ may be smaller than $1$ for some $t >0$ and $x \in X$, one adds to it the  amount of heat which was removed by the killing 	term up to time $t$. Therefore, 
the  condition $M_t(x)=1$ means that the amount of heat lost from the graph is exactly the amount removed by the killing term. We  also refer to \cite[Section~7.5]{KLW} for a more detailed discussion of this concept and its relation to the classical concept of stochastic completeness.

We furthermore recall the probabilistic interpretation of the semigroup for   a graph $b$ without killing term  in terms of the associated Markov process $(\mathbb{X}_t)$ with lifetime $\zeta$, see \cite[Chapter~2.5]{KLW}. For a graph $ (b,c) $, we then have the following Feynman-Kac formula, \cite[Theorem~2.31]{KLW} for the semigroup 
\begin{align*}
e^{-tL}f(x) = \mathbb{E}_x\Big(1_{\{ \zeta > t \}} e^{-\int_0^t c(\mathbb{X}_s) ds} f(\mathbb{X}_t)\Big).
\end{align*}

We will prove the following theorem which shows the equivalence of (ii) and (iii) in Theorem~\ref{thm:main} and also gives a further probabilistic interpretation of the condition in (iii).

\begin{theorem}[The stochastic completeness perspective]
		\label{thm:Heat}
Let $(b,c)$ be a connected transient graph over $X$. 
Then, the following assertions are equivalent.
\begin{itemize}
\item[(ii.a)] $Gc(x) = 1$ for some (all) $x \in X$.
\item[(iii.a)]  The graph is stochastically complete at infinity and for some (all) $x \in X$ 
$$\lim_{t\to\infty} e^{-tL}1(x) =\lim_{t\to\infty} \sum_X    e^{-tL}1_x= 0.$$
\item[(iii.b)]   We have for some (all) $x \in X$, 
\begin{align*}
	\mathbb{P}_x \Big( \int_0^{\zeta} c(\mathbb{X}_t) dt = \infty \Big) = 1.
\end{align*}
\end{itemize}	
\end{theorem}
We start by showing a formula relating $ Gc $ with the limit of the difference of  $ M_t $ and $ e^{-tL}1 $ as $ t \to \infty $.
\begin{lemma}
	\label{lem:heat} Let $(b,c)$ be a connected transient graph over $X$.  Then, for all $x \in X$,
\begin{align*}
	Gc(x) =\lim_{t \to \infty} M_t(x)  - \lim_{t \to \infty} e^{-tL}1(x),
\end{align*}
where both limits on the right-hand side exist.
\end{lemma}
\begin{proof}
We observe by the semigroup representation of the Green function and Fubini's theorem, we have 
	\begin{multline*}
		 ( Gc)(x) =\sum_{y\in X} G(x,y) c(y) 
		 =\sum_{y\in X} \int_0^\infty e^{-sL}1_y(x) c(y) ds\\
		= \int_0^\infty \sum_{y\in X} e^{-sL}1_y(x) c(y) ds
		=\int_0^\infty \big( e^{-sL}  {c} \big)(x)\, ds\\
		= \lim_{t \to \infty} \int_0^t \big( e^{-sL}  {c} \big)(x)\, ds
		=\lim_{t \to \infty}\Big( M_t(x)  -  e^{-tL}1(x)\Big)
	\end{multline*}
	by definition of $ M_t $. Furthermore,  the function $ t\mapsto M_t(x) $ is decreasing,  bounded by $1$, and non-negative for all $x \in X$, \cite[Theorem~7.14]{KLW}. Therefore, the limit $ \lim_{t \to \infty} M_t(x) $ exists. As the limit of the difference exists, the limit $ \lim_{t \to \infty} e^{-tL}1(x) $ must exist as well. This allows us to conclude the above formula for $ Gc(x) $.
\end{proof}

We prove a second lemma which yields the interpretation that $ e^{-tL}1(x) $ represents the mass of heat which is still within the graph at time $ t $ with an initial distribution of $ 1_x $ at time $ t=0 $.

\begin{lemma}
\label{lem:heat2} Let $(b,c)$ be a connected  graph over $X$ and let $ (\mathbb{X}_t) $ be the Markov process associated with $b$ and lifetime $\zeta$.  Then, for all $x \in X$ and $t >0$,
\begin{align*}
e^{-tL}1(x) = \sum_{X}e^{-tL}1_x  = \mathbb{E}_x\Big(1_{\{\zeta > t\}}e^{-\int_{0}^{t}c(\mathbb{X}_s)}ds  \Big).
\end{align*}
\end{lemma}
\begin{proof}
	The first formula is an easy consequence of the duality of the semigroup on $\ell^1(X)$ and $\ell^\infty(X)$, see \cite[Chapter~2.1]{KLW}, i.e., since $ 1\in \ell^{\infty}(X) $ and $ 1_x\in \ell^{1}(X) $, we have
	\begin{align*}
e^{-tL}1(x) &=\langle e^{-tL}1, 1_x\rangle_{\ell^{\infty},\ell^{1}} = \langle 1, e^{-tL}1_x\rangle_{\ell^{\infty},\ell^{1}} = \sum_{X}e^{-tL}1_x.
	\end{align*}
	The second formula is a direct consequence of the Feynman-Kac formula and was shown in  \cite[Theorem~7.32]{ KLW}.
\end{proof}

\begin{proof}[Proof of Theorem~\ref{thm:Heat}]
The equivalence for ``some'' and ``all'' $x \in X$ in (ii.a) was already shown in Lemma~\ref{cor:Gc_dichotomy} and will imply the corresponding  for ``some'' and ``all'' equivalences in (iii.a) and  (iii.b) as well. Hence, it suffices to show the equivalence for a fixed $x \in X$.  
\medskip

(ii.a) $ \Leftrightarrow $ (iii.a): We observe that  $ 0\le e^{-tL}1\leq M_t\leq 1 $.
Hence,  Lemma~\ref{lem:heat} gives the desired equivalence and the equality in (iii.a) was shown in Lemma~\ref{lem:heat2}.\medskip


(iii.b) $ \Leftrightarrow $ (iii.a): 
By \cite[Theorem~7.33]{KLW}, we have that 
$$ \mathbb{P}_x(\zeta=\infty)=1 \quad \mbox{or} \quad 
\mathbb{P}_x\Big(\int_{0}^{\zeta}{c(\mathbb{X}_t)}dt =\infty \mid \zeta < \infty\Big)=1$$
for some (all) $x \in X$ 
is equivalent to  stochastic completeness at infinity.
Furthermore, by Lemma~\ref{lem:heat2} and Lebesgue's dominated convergence theorem, we have for all $x \in X$ that
\begin{align*}\lim_{t \to \infty} e^{-tL}1(x) &= \lim_{t \to \infty} \mathbb{E}_x\Big(1_{\{\zeta > t\}}e^{-\int_{0}^{t}c(\mathbb{X}_s)}ds  \Big) = \mathbb{E}_x\Big(1_{\{\zeta = \infty\}}e^{-\int_{0}^{\infty}c(\mathbb{X}_s)}ds  \Big).
\end{align*}
This readily gives that 
(iii.b) implies (iii.a). For the converse direction, note that the condition on the expectation to be $0$ yields
\[
\mathbb{P}_x\Big( \int_0^{\zeta} c(\mathbb{X}_t)  dt = \infty  \mid  \zeta = \infty \Big) = 1
\]
or $\mathbb{P}_x(\zeta = \infty) = 0$ (in which case there is nothing else to show).
This together with the theorem on total probability,
\begin{align*}
	\mathbb{P}_x \Big( \int_0^{\zeta} c(\mathbb{X}_t) dt = \infty \Big) &= \mathbb{P}_x (\zeta = \infty) \mathbb{P}_x\Big( \int_0^{\zeta} c(\mathbb{X}_t)  dt = \infty \mid  \zeta = \infty \Big) \\
	& \quad  + \, \mathbb{P}_x (\zeta < \infty) \mathbb{P}_x\Big( \int_0^{\zeta} c(\mathbb{X}_t)  dt = \infty \mid  \zeta < \infty \Big)
\end{align*}
yields the claim in both cases $\mathbb{P}_x(\zeta=\infty) =1$ or $\mathbb{P}_x(\int_0^{\zeta} c(\mathbb{X}_t)\, dt= \infty \mid \zeta < \infty) = 1$. 
\end{proof}

\begin{remark} Looking at Lemma~\ref{lem:heat} from the perspective of the Riesz decomposition of the constant function $1=u_p+u_h$ gives that for graphs which are stochastic complete at infinity, i.e.,\@ $M_t = 1$ for all $t >0$ that
\begin{align*} u_h = \lim_{t \to \infty} e^{-tL}1.\end{align*}
\end{remark}

We further draw a corollary from the above theorem which shows that increasing $c$ preserves the Liouville property. 
\begin{corollary}\label{cor:Gc_increase}
	Let $(b,c)$ be a connected transient graph over $X$. If $G^{b,c}c=1$, then $G^{b,c+c'}(c+c')=1$ for all $c' \geq 0$.
\end{corollary}
\begin{proof}This follows directly from (iii.b) of the above theorem.
\end{proof}

We finish this section with a consequence of the above theorem which shows that the condition $G_o \in \ell^1(X)$ implies the Liouville property.

\begin{proposition}
	Let $(b,c)$ be a connected transient graph over $X$ which is stochastically complete at infinity	 and $o \in X$ such that $G_o \in \ell^1(X)$. Then, every bounded harmonic function is zero.
\end{proposition}
\begin{proof}
	 Assuming that $G_o \in \ell^1(X)$, we have
	 by the semigroup representation of the Green function, \cite[Theorem~6.29]{KLW}, that
	 $${\|G_o\|}_1=\sum_{X} G_o =\sum_{X} \int_0^\infty e^{-tL}1_o \, dt=\int_0^\infty e^{-tL}1(o) \, dt, $$ 
	 where the third equality is Fubini's theorem and   Lemma~\ref{lem:heat2}. As the limit $\lim_{t \to \infty} e^{-tL}1(o)$ exists by Lemma~\ref{lem:heat}, this equality implies that $e^{-tL}1(o)$ tends to zero as $t \to \infty$ whenever the $ \ell^{1} $-norm is finite. Hence, if $G_o \in \ell^1(X)$, then Theorem~\ref{thm:main} implies the statement.
\end{proof}

\section{The Green formula perspective}\label{sec:GreenFormula}

In this section we give the proof of the equivalence of (ii) and (iv) in Theorem~\ref{thm:main}. Specifically, we show that the Liouville theorem is equivalent to the validity of a Green formula for potentials.

We denote the space
\begin{align*}
\ell^{1}(X,c) = \{f\colon X \to \R \mid \sum_X c|f| < \infty\}.
\end{align*}

\begin{theorem}[Green formula perspective]
		\label{thm:GF}
Let $(b,c)$ be a connected transient graph over $X$. 
Then, the following assertions are equivalent.
\begin{itemize}
\item[(ii.a)] $Gc(x) = 1$ for some (all) $x \in X$, i.e., $1\in\P$.
\item[(iv.a)] For some  nontrivial superharmonic $u\in\P\cap \ell^1(X,c)$ (respectively, for all superharmonic $u\in\P$ or $ u\in \D_0 $), we have 
$$\sum_X  \mathcal{L} u = \sum_X cu,$$
where both sides sum over positive terms and may  be  infinite.
\item[(iv.b)] A positive superharmonic function $u \in \ell^1(X,c)$ is a potential if and only if $$\sum_X  \mathcal{L} u = \sum_X cu,$$
where both sides converge absolutely.
\item[(iv.c)]   For some (all)  $x\in X$, one has $$ 1= \sum_X cG_x.$$
\end{itemize}	
\end{theorem}

We start by showing a Green formula for potentials, which will be the main ingredient for the proof of the above theorem. 
\begin{lemma}\label{lem:GreenFormula}
	If $u, v$ are superharmonic potentials, then 
	\begin{align*}
		\sum_X v \L u = \sum_X u \L v,
	\end{align*}
	where both sides may be infinite.  In particular, for every superharmonic potential $u$,
	\begin{align*}
		\sum_X (Gc) \L u = \sum_X c u.
	\end{align*}
\end{lemma}
\begin{proof}
	Let $u = Gk$, $v = Gl$, $k, l \in \G$, $k, l \geq 0$. Then by Fubini's theorem and the symmetry of $G$, 
	\begin{align*} 		\sum_X v \L u = \sum_{X} k G l =  \sum_{x ,y\in X} k(x) G(x,y) l(y)=  \sum_{ X} lGk = \sum_X u \L v. 	\end{align*}
	The ```in particular'' statement follows from setting $v = Gc$.
\end{proof}

Next, show that the condition in (iv.a) implies that every positive superharmonic function in $\ell^1(X,c)$ is a potential.
\begin{lemma}
	\label{lem:u_is_potential}
	If $c \neq 0$ and $u \in \ell^1(X,c)$ is a positive superharmonic function with 
	\begin{align*}
		\sum_X (Gc) \L u \geq \sum_X cu,
	\end{align*}
	then $u$ is a potential. 
\end{lemma}
\begin{proof}
	Let $u \in \ell^1(X,c)$ be a positive superharmonic function with $		\sum_X (Gc) \L u \geq \sum_X cu.$
	Then $u$ has a Riesz decomposition $u = u_p + u_h$ with $0\le u_p \leq u$. It follows by the ``in particular'' statement of the lemma above that
	\begin{align*}
		\sum_X cu_p \leq \sum_X cu \leq \sum_X (Gc) \L u = \sum_X (Gc) \L u_p = \sum_X c u_p.
	\end{align*}
	Therefore, $\sum_X c u_p = \sum_X c u < \infty$. Since  $c \neq 0$, we infer that the positive harmonic function $u_h= u-u_p $ vanishes at some point (i.e, at some $x\in X$ with $c(x) \neq 0$) and, therefore, vanishes everywhere by the Harnack inequality	 \cite[Lemma~4.5]{KLW}. Hence, $u = u_p$ is a potential.
\end{proof}

\begin{proof}[Proof of Theorem~\ref{thm:GF}]
The implication (ii.a) $ \Rightarrow $ (iv.a)   for all superharmonic $ u\in \P $ immediately follows from the above lemma.  
  \medskip


Assume (iv.a) holds for some non-trivial superharmonic $ u=Gk\in\P \cap \ell^1(X,c) $. By $\L u=k\ge 0$ and Lemma~\ref{lem:GreenFormula}, we have
\begin{align*}\sum_X k =\sum_X  \mathcal{L} u= \sum_X cu = \sum_X  (Gc)  \mathcal{L}u=
	   \sum_X  (Gc) k.
\end{align*}
Hence, $Gc=1$ for all $x\in X$ such that $k(x)\ne0$. Hence, $Gc=1$ for all $x\in X$ by Lemma~\ref{cor:Gc_dichotomy} which is (ii.a). \medskip

Furthermore,  all superharmonic $u \in \mathcal{D}_0$ are potentials, cf.\@ \cite[Lemma~8]{HKPII}. On the other hand, $G_x$, $x \in X$,  satisfies by Fubini's theorem $ \sum_X cG_x = (Gc)(x) \leq 1 $ and  consequently $ G_x\in \ell^1(X,c) $. Since  $G_x$ is also superharmonic and in $ \D_0$  by \cite[Theorem~6.26~and~6.29]{KLW}, there exist  non-trivial superharmonic function in $\mathcal{D}_0\cap \ell^1(X,c)$ which gives the equivalence in (iv.a) with respect to $\D_0$. \medskip

We next show that (iv.a) implies (iv.b). Let $u \in \ell^1(X,c)$ be positive and superharmonic. If $u$ is a potential then $\sum_X \mathcal{L} u = \sum_X cu$ holds by (iv.a). On the other hand, we have already shown that (iv.a) implies (ii.a), which gives $c\neq0$. So, if $\sum_X cu = \sum_X \mathcal{L} u$ then $u$ is a potential by Lemma~\ref{lem:u_is_potential}.  \medskip

The implication (iv.b) to (iv.a) is clear by Lemma~\ref{lem:GreenFormula}, as $G_x\in \P\cap \ell^1(X,c)$ for  $x \in X$. \medskip

The equivalence of (iv.c) and (ii.a) follows from Fubini's theorem.
\end{proof}

\section{The decomposition perspective}\label{sec:Decomposition}
In this section we decompose the killing term $c$ into a part which is supported on a non-thick set and a part for which the free Green function applied to it stays bounded. We show that the Liouville property is equivalent to the non-existence  of such a decomposition.

\begin{theorem}\label{thm:decomposition}
Let $(b,c)$ be a connected transient graph over $X$. 
Then, the following assertions are equivalent.
\begin{itemize} 
 \item[(ii)]  $G c = 1$.
 \item[(v)]   For any non-thick $A \subseteq X$ and some (all)  $x \in X$
  $$\sum_{y \in X\setminus A} G^b(x,y) c(y) = \infty. $$
 \item[(v.a)] There exists no decomposition $c = c_1 + c_2$ with $c_1, c_2 \geq 0$, such that $c_1$ is supported on some non-thick set  and $G^{b }c_2 $ is bounded.
  \item[(v.b)] There exists no decomposition $c = c_1 + c_2$ with $c_1, c_2 \geq 0$, such that $c_1$ is supported on some non-thick set  and $G^{b }c_2 $ is finite.

\end{itemize}
In $\mathrm{(v.a)}$ and $\mathrm{(v.b)}$, it suffices to exclude decompositions of the form $c_1 = 1_A c$ and $c_2 = 1_{X \setminus A} c$ for some non-thick $A \subseteq X$.
\end{theorem}

We start by showing a resolvent formula for the Green function of a graph with killing term.

\begin{lemma}\label{lemma:resolvent}
Assume that $c,\tilde c \colon  X \to [0,\infty)$ with $ \tilde c \leq  c$. If $f$ is nonnegative, then
$$G^{b,\tilde c}f = G^{b,c}f + G_n^{b,\tilde c}((c-\tilde  c) G^{b, c} f).$$
In particular, $G^{b,\tilde c}f \geq G^{b, c}f$ and if $f \in \mathcal G^{b,\tilde c}$,   then
$$G^{b,c} f =  G^{b,\tilde c}(f - ( c -\tilde c)G^{b,  c} f). $$
Moreover, we have
$$G^{b,c} c \geq G^{b,0}(c (1 -G^{b,  c} c)).$$
%
\end{lemma}
\begin{proof}Let $(K_n)$ be an exhaustion of $X$ and let $G_n^{b,c}$ and $G_n^{b,\tilde c}$ be the Green function corresponding to the graph $(b,c)$ and $(b,\tilde c)$  restricted to $K_n$ with Dirichlet boundary conditions, extended by $0$ to $X \setminus K_n$. Moreover, let $c_n=c1_{K_n}$ and $\tilde{c}_n=\tilde{c}1_{K_n}$. Then by the resolvent formula one has
\begin{align*}
	G_n^{b,\tilde c}-G_n^{b,c} = G_n^{b,\tilde c}(c_n-\tilde  c_n) G_n^{b, c}.
\end{align*}
By \cite[Proposition~1.20]{KLW} these restricted Green functions are monotone increasing in $n$. Moreover, \cite[Theorem~6.26]{KLW} shows $G_n^{b,c} \to G^{b,c}$ and $G_n^{b,\tilde c} \to G^{b,\tilde c}$ pointwise, as $n \to \infty$. Hence, applying the induced Green operators to $f_n=f1_{K_n}$ yields the first equality with the help of the monotone convergence theorem. This also implies $G^{b,\tilde c}f \geq G^{b, c}f$.

The second identity follows from rearranging the first noting that  all the involved infinite sums are absolutely convergent because $f \in  \mathcal G^{b,\tilde c}$ and the first equality holds for $|f|$.

The 'Moreover'-statement follows from the formula for the approximating Green functions with $\tilde c = 0$ and $f = c$ and Fatou's lemma. It can be applied, because all of the summands of $G^{b,0}_n(c_n(1-G^{b,c_n}c_n))$ are nonnegative.
\end{proof}

Next, we show that the Liouville property is invariant under multiplication of the killing term by a positive constant.  
\begin{lemma} \label{lemma:multiplication by constants}
 Let $(b,c)$ be a graph. Then  $Gc = 1$ if and only if 
 $G^{b,\lambda c} (\lambda c) = 1$ for some (all) $\lambda > 0$. 
\end{lemma}
\begin{proof}
It suffices to show that if $G^{b,\lambda c} (\lambda c) = 1$ for some  $\lambda > 0$, then $G^{b,\mu c} (\mu c) = 1$ for all $\mu > 0$. So, assume   $G^{b,\lambda c} (\lambda c) = 1$ for   $\lambda > 0$  and let $\mu > 0$.

If $\lambda > \mu$, then the previous lemma, Lemma~\ref{lemma:resolvent}, applied to  $\mu c \leq \lambda c$ and $G^{b,\lambda c}(\lambda c) = 1$ yield
\begin{align*}
 1 = G^{b,\lambda c} (\lambda c) =  G^{b,\mu c}(\lambda c - (\lambda c - \mu c)G^{b,\lambda c} (\lambda c))
 = G^{b,\mu c}(\mu c).
\end{align*}

If $\lambda < \mu$, then Lemma~\ref{cor:Gc_dichotomy} gives  $G^{b,\mu c} (\mu c) \leq 1$ and Lemma~\ref{lemma:resolvent},	 applied to  $\lambda c \leq \mu c$  yield 
\begin{multline*}
 1 \geq G^{b,\mu c} (\mu c)  =  G^{b,\lambda c}(\mu c - (\mu c - \lambda c)G^{b,\mu c} (\mu c)) \\
  = \frac{\mu}{\lambda} - (\mu - \lambda) G^{b,\lambda c} (c G^{b,\mu c} (\mu c))
  \ge \frac{\mu}{\lambda} - (\mu - \lambda) G^{b,\lambda c} c =1,
\end{multline*}
where we applied $G^{b,\lambda c}( cG^{b,\mu c} (\mu c)) \leq G^{b,\lambda c} c $ and 
$G^{b,\lambda c}( c) = 1/\lambda$ in the last steps. Hence, we obtain $G^{b,\mu c} (\mu c) = 1$.
\end{proof}

\begin{remark}
	Alternatively, the lemma also follows from Theorem 8 (iii.b), similar to Lemma~\ref{cor:Gc_increase}.
\end{remark}

\begin{proof}[Proof of Theorem~\ref{thm:decomposition}]
(v.a) $\Rightarrow$ (ii): We prove the statement by contraposition, so, assume $Gc\neq 1$.  Let $s = 1 - G  c$ and for $0 < \alpha < 1$ consider the set $A = \{x \in X \mid s(x) < \alpha\}$.

\emph{$A$ is not thick:} The nonnegative superharmonic function $Gc$ satisfies $Gc=1-s \geq 1 - \alpha$ on $A$. If $A$ were thick, we would obtain $Gc \geq 1 - \alpha$ on $X$. But  $G c \neq 1$ implies $\inf_{  X} Gc  = 0$ by Lemma~\ref{cor:Gc_dichotomy}, which contradicts $Gc \geq 1 - \alpha$.

\emph{$G^{b,0} c_2$ bounded:} Now, let $c_1 = 1_A c$ and $c_2 = 1_{X \setminus A} c$. Using Lemma~\ref{lemma:resolvent} and the definition of $A$, we obtain
\begin{align*}
G^{b,0} c_2 = G^{b,0}  (1_{X \setminus A} c) \leq \frac{1}{\alpha} G^{b,0} (c s) \leq  \frac{1}{\alpha} G^{b,c} c \leq \frac{1}{\alpha}.
\end{align*}

(ii) $\Rightarrow$ (v.b): By contraposition, we choose a decomposition $c = c_1 + c_2$ as in (v.b).

Since $A$ is non-thick, we find a nonnegative superharmonic function $s$ satisfying $s \geq 1$ on $A$ but with $\inf_{  X} s < 1$. Without loss of generality, we can assume $s \leq 1$ (else consider $s \wedge 1$) such that $s = 1$ on $A \supset {\rm supp}\,(c_1)$. Then
$$\mathcal L_{b,c_1}(1 - s) = c_1 - s c_1 - \mathcal{L}_{b,0}s = - \mathcal{L}_{b,0}s \leq 0. $$
Lemma~\ref{cor:Gc_dichotomy} (applied to the graph $(b,c_1)$ and $u = 1-s$) implies $G^{b,c_1} c_1 \leq 1 - (1-s) = s$.  Since $\inf_{  X} s  < 1$, we infer $\inf_{  X} G^{b,c_1} c_1  = 0$ with the help of Lemma~\ref{cor:Gc_dichotomy}.  The same argument holds for $\lambda c_1$ instead of $c_1$, showing $G^{b,\lambda c_1}(\lambda c_1) \leq s$ for any $\lambda > 0$.

With this at hand, Lemma~\ref{lemma:resolvent} implies that for any $\lambda > 0$ we have
$$G^{b,\lambda c}(\lambda c )  = G^{b,\lambda c}(\lambda c_1)  + G^{b,\lambda c}(\lambda c_2)  \leq G^{b,\lambda c_1}(\lambda c_1)  + G^{b,0}(\lambda c_2) \leq s +\lambda G^{b,0} c_2  .$$
Since $\inf_{  X} s < 1$ and $G^{b,0} c_2 < \infty$, we find $x \in X$ and $\lambda > 0$ with $s(x) + \lambda G^{b,0} c_2(x) < 1$. Hence, for this particular $\lambda$,  Lemma~\ref{cor:Gc_dichotomy} implies $G^{b,\lambda c}(\lambda c) < 1$. Now, Lemma~\ref{lemma:multiplication by constants} yields $G^{b,c}c < 1$.\medskip

(v.b) $\Rightarrow$ (v.a): This is clear.\medskip

The choice of $c_1$ and $c_2$ in (v.a) and (v.b)  as claimed follow from the construction above. As  $G^{b} (1_{X \setminus A} c)(x) = \infty$ for one $x \in X$ if and only if $G^{b,0} (1_{X \setminus A} c)(x) = \infty$ for all $x \in X$, which is implied by $G(x,y) \geq C G(z,y) $ for all $x,y,z \in X$ and some constant $C=C(x,z) > 0$, cf.~\cite[Proof of Theorem~6.26]{KLW}. This gives the equivalence of (v) and (v.b).
\end{proof}

One may wonder whether the decomposition $c = c_1 + c_2$ is indeed needed or if $G^bc=\infty$ already implies $Gc=1$. The following example shows that this is not the case. In particular, the condition $G^bc=\infty$ does not imply the Liouville property.

\begin{example}[$G^bc=\infty$ $\not\Rightarrow$ $Gc=1$]
	Let $X = \Z$, $x \sim y$ iff.\@ $|{x-y}| = 1$ and $b(k-1,k) = b(k,k-1) = 2^{k_+}$ where $k_+ = \max(0,k)$. Then $G^b_0(k) = 2^{-k_+}$ is the Green function at $0$ for the graph $b$. Let $c = 1_{-\N}$ and let $u = 1-G^b_0$. Then
	\begin{align*}
		\L^{b,c} u = \L^b (1-G^b_0) + 1_{-\N} (1-G^b_0) = - 1_{\{0\}}
	\end{align*}
	and $u$ is a non-trivial positive bounded subharmonic function on $(b,c)$. But we also have 
	\begin{align*}
		(G^b c)(0) = \sum_{x \in X} G^b(0,x) c(x) = \sum_{n=1}^\infty G^b_0(-n) = \infty. 
	\end{align*}
In view of Theorem~\ref{thm:decomposition}, note that here we have the decomposition $c=c_1+c_2$ with $c_1 = 1_{-\N}c = 1_{-\N}$ and $c_2=1_{\N_0}c = 0$, $G^{b,0}c_2 = 0$ and $-\N$ is a non-thick set since $s=G_0^b$ is $1$ on $-\N$ but strictly smaller than one on $\N$.
\end{example}

We next characterize thick sets. As thickness is defined in terms of $b$ only we consider the graph $(b,0)$ in the following. We will see that thickness of $X \setminus A$ is equivalent to the Liouville property for the graph $(b_A,d_A)$, where $d_A$ is the potential corresponding to Dirichlet boundary conditions on $A$ introduced in Section~\ref{sec:intro}. We denote the Green function of $(b_A,d_A)$ by $G^A$.

 Observe, that for $\emptyset \neq U \subseteq X$,  the graph $(b_U,d_U)$ is not necessarily  connected. However,  we can still apply our theorems by considering each connected component of $(b_U,d_U)$ and  since $(b,0)$ is connected,  $d_U$ does not vanish identically on any connected component. Hence, the restriction of $(b_U,d_U)$ to a connected component is a transient graph with nonvanishing potential.

Before we come to the characterization of thick sets, we need the following observation.

\begin{lemma}\label{lem:stochcompl}
 If $(b,0)$ is stochastically complete, then $(b_U,d_U)$ is stochastically complete at infinity for any   $\emptyset \neq U \subseteq X$.
\end{lemma}
\begin{proof}
 Let $\alpha > 0$ and let $u \colon U \to [0,\infty)$ be bounded with $(\mathcal L_{b_U,d_U} + \alpha)u \leq 0$ and let $v = u$ on $U$ and $v = 0$ on $X \setminus U$.  For $x \in U$, we obtain
 $$(\mathcal L_{b,0} + \alpha)v(x) = (\mathcal L_{b_U,d_U} + \alpha)u(x) \leq 0$$
 and for $x \in X \setminus U$ we have
 $$(\mathcal L_{b,0} + \alpha)v(x) = -\sum_{y \in X}b(x,y) v(y) \leq 0.$$
 Since $(b,0)$ is stochastically complete, we infer $v = 0$ by \cite[Theorem~7.2]{KLW} and hence $u = 0$, showing that $(b_U,c_U)$ is stochastically complete at infinity by \cite[Theorem~7.18]{KLW}.
\end{proof}

\begin{proposition}\label{proposition:characterization thickness}
 Let $A \subseteq X$. The following assertions are equivalent:
 \begin{itemize}
  \item[(i)] $X \setminus A$ is  thick.
 \item[(ii)] $G^A d_A =  1$.
 \item[(iii)] Any bounded function that is harmonic on $A$ and vanishes on $X \setminus A$ equals zero.
 \item[(iv)] $(b_A,d_A)$ is stochastically complete at infinity and
  $$ e^{-t L_{A}} 1 \to 0, \text{ as $t\to\infty$} \text{   pointwise on }A.$$
 \end{itemize}
If $(b,0)$ is stochastically complete, then the assertions above are equivalent to the following:
\begin{itemize}
  \item[(v)]  $$ e^{-t L_{A}} 1 \to 0, \text{ as $t\to\infty$} \text{   pointwise on }A.$$
 
 \end{itemize}

\end{proposition}
\begin{proof}
The equivalence of (ii),(iii) and (iv) is contained in   Theorem~\ref{thm:main} applied to the graph $(b_A,d_A)$.\medskip

(i) $\Rightarrow$ (ii): Assume $G^A d_A \neq    1$ implying $G^A d_A < 1$ by Lemma~\ref{cor:Gc_dichotomy}. We let
$$s_A \colon X \to \R,\quad s_A(x) = \begin{cases}
                                     G^A d_A(x) &\text{if } x\in A,\\
                                     1 &\text{if } x \in X \setminus A.
                                     \end{cases}
$$
Then $0 \leq s_A \leq 1$ by Lemma~\ref{cor:Gc_dichotomy} and, for $x \in A$,
we have
\begin{align*}
 \mathcal L_{b,0} s_A(x)  &= \sum_{y \in A} b(x,y)(G^A d_A(x) - G^Ad_A(y)) + d_A(x) (G^A d_A (x) - 1) \\
 &=\mathcal L_{b_A,d_A} (G^A d_A)(x) - d_A(x) = d_A(x) - d_A(x) = 0.
\end{align*}
Moreover, for $x \in X \setminus A$, we have
$$\mathcal L_{b,0} s_A(x) =   \sum_{y \in X} b(x,y) (1 - s_A(y)) \geq 0.$$
Hence, $s_A$ is positive, bounded and superharmonic and satisfies $s_A \geq 1$ on $X \setminus A$. But $s_A \geq 1$ on $X$ fails because $G^A d_A < 1$, showing that $X \setminus A$ is not thick.\medskip

(ii) $\Rightarrow$ (i): Assume that $X \setminus A$ is not thick. Choose a positive superharmonic function $s$ with $s \geq 1$ on $X \setminus A$ but $\inf_A s < 1$. Without loss of generality assume $s \leq 1$ (else consider $s \wedge 1$), which yields $s = 1$ on $X \setminus A$. Then $1 - s$ is subharmonic and   $u_A = (1 - s)|_A$  satisfies
$$\mathcal L_{b_A,d_A} u_A(x) = \mathcal L_{b,0}(1-s)(x) \leq 0$$
for all $x \in A$. Lemma~\ref{cor:Gc_dichotomy} yields $G^A d_A \leq 1 - u_A = s|_A$. Since $\inf_A s < 1$, we obtain $G^A d_A < 1$.\medskip

(iii) $\Leftrightarrow$ (iv): This follows from Lemma~\ref{lem:stochcompl}.
\end{proof}

\begin{remark}
 The characterization of thick sets (v) shows that for stochastically complete graphs $(b,0)$ the condition (v.b) is the same as in Batty's theorem on Schrödinger operators on $\R^d$, \cite{Bat92}.
\end{remark}

We finally put the pieces together and give the proof of the main result, Theorem~\ref{thm:main}. 

\begin{proof}[Proof  of  Theorem~\ref{thm:main}] 
	The equivalence of (i) and (ii) was shown in Theorem~\ref{thm:Liouville}. \\
	The equivalence of (ii) and (iii) was shown in Theorem~\ref{thm:Heat}. \\
	The equivalence of (ii) and (iv) was shown in Theorem~\ref{thm:GF}.  \\
	The equivalence of (ii) and (v) was shown in Theorem~\ref{thm:decomposition}.  
\end{proof}

\section{Recurrent and transient graphs without killing term}\label{sec:Rec}

In this section we study the validity of the Liouville theorem  by considering the graph $b$ of $(b,c)$ separately from the killing term $c$. We will see that the Liouville theorem holds if  either $b$ is recurrent or $c$ is summable.

We start with an example  which illustrates this phenomenon.
\begin{example} Consider the  Laplacian with standard weights on $\N$ with a Dirichlet boundary condition at $0$, i.e., $ b(x,y)=1 $ if  $ |x-y|=1 $ and  $b (x,y) = 0 $ otherwise and $c=1_{\{1\}}$. Clearly, the graph $b$  is recurrent. Furthermore, the Green function at $ 1 $ for the graph $(b,c)$ is given by $G_1(x) = 1$ for all $x \in \N$, \cite[Exercise~6.20]{KLW}. We conclude   $$Gc(1)=  c(1)=1.$$ 
It is also clear that the only harmonic functions with Dirichlet boundary condition at $0$ are  linear. Hence, the Liouville theorem holds for this graph.
\end{example}

\begin{theorem}
	\label{thm:recurrence} Let $(b,c)$ be connected graph over $X$. 
	\begin{itemize}
	\item[(a)] If $b$ is recurrent and $c \neq 0$, then every bounded harmonic function is zero.
	\item[(b)] If  $b$ is transient and $c\in\ell^1(X)$, then there exist non-zero bounded harmonic functions.
	\end{itemize}
\end{theorem}

We prove the two implications separately in the following two lemmas. 
\eat{\begin{lemma}
\label{lem:recurrence} Let $(b,c)$ be connected graph over $X$ and $c\neq 0$.  
If the graph $b$ is recurrent, then $Gc=1$.
\end{lemma}
\begin{proof}
We first assume $0\neq c\in C_c(X)$, i.e., $c$ is finitely supported. If $ b $ is recurrent, then $ 1\in \D_0(\mathcal{Q}_b) $ and, therefore, $ 1\in \D_0(\mathcal{Q}_{b,c}) $. 
Thus, since $ c=\mathcal{L}1 $, we obtain  by the Green formula in $\mathcal{D}_0$, \cite[Lemma~6.8]{KLW}, applied twice 
\begin{align*}
\sum_{X} c G_{o} = \sum_X (\mathcal{L}1) G_{o} = \mathcal{Q}(1,G_{o}) = \sum_X  \mathcal{L}G_{o}=1.
\end{align*}
Thus, the statement follows from   Theorem~\ref{thm:GF}~(iv.c).
Now, for any $c'\ge c \neq 0$, let $G'$ be the Green function corresponding to the graph $(b,c')$. Then, then $G'c'=1$,  follows immediately from Theorem~\ref{thm:Heat}~(iii.c).
\end{proof}
}

\begin{lemma}\label{lem:recurrence1} Let $(b,c)$ be connected graph over $X$.
	If $b$ is recurrent and $c \neq 0$ then every bounded positive subharmonic function is zero.
\end{lemma}
\begin{proof}
	Let $u$ be a bounded positive subharmonic function. Then 
	\begin{align*}
		\L_{b,0} (-u) = \L_{b,c}(-u) - c(-u) = - \L_{b,c} u + cu \geq 0, 
	\end{align*}
	and so $-u$ is bounded and $\L_{b,0}$-superharmonic. Since $b$ is recurrent this implies that $-u$ is constant. But a strictly positive constant is not $\L_{b,c}$-subharmonic due to $c\neq 0$ and we must have $u = 0$.
\end{proof}

\begin{lemma}\label{lem:recurrence2} Let $(b,c)$ be connected graph over $X$.  If $Gc=1$, then $G^bc=\infty$. In particular, if $b$ is transient and $c\in \ell^1(X)$, then there are non-zero bounded harmonic functions.
\end{lemma}
\begin{proof} We apply Lemma~\ref{lemma:resolvent} with $\tilde c = 0$  to obtain for $Gc=1$, $G^bc<\infty$ and $f=c$ that
\begin{align*}
	1=Gc=G^bc-G^bcGc=G^bc-G^bc=0,
\end{align*}
which is a contradiction. Hence, $Gc=1$ implies $G^bc=\infty$. 

For transient  $b$, the Green function $G^b$ exists and $G^b_x$ is bounded for all $x\in X$ by  \cite[Theorem~B.1]{KLSS}. Thus,  $c\in \ell^1(X)$, implies $G^bc<\infty$ by Hölder's inequality and, therefore, $Gc=1$ cannot hold. Hence, there are non-zero bounded harmonic functions by Theorem~\ref{thm:main}.
\end{proof}

\begin{proof}[Proof of Theorem~\ref{thm:recurrence}] The statement follows immediately from the two lemmas above and Theorem~\ref{thm:main}.	
\end{proof}

\begin{remark}
  \cite[Theorem 3.1]{GO96} states that an irreducible Dirichlet form is recurrent if and only if its semigroup is not $L^1$-asymptotically stable but any perturbation by a nonnegative potential has an $L^1$-asymptotically stable semigroup. Applied to our Dirichlet form $Q$ it can be phrased as follows: the graph $b$ is recurrent if and only if $e^{-tL_{b,0}} 1 \not \to 0$, as $t \to \infty$, and  $e^{-tL_{b,c}} 1  \to 0$, as $t \to \infty$, for any non-trivial $c $. 
  Since recurrent graphs $b$ are stochastically complete and, therefore, $(b,c)$ is stochastically complete at infinity Theorem~\ref{thm:main}~(iii) gives therefore Theorem~\ref{thm:recurrence}~(b) as a special case.
 \end{remark}

\section{The case of positive Schrödinger operators}\label{sec:measure}

In this section we  discuss the case of Schrödinger operators on general measure spaces. 
We briefly recall the setting and refer to \cite{KPP20} for a more detailed discussion.

A strictly positive function $ m\colon X \to (0,\infty) $ extends to a measure of full support on $X$ by setting $$  m(A) = \sum_{x\in A} m(x)  $$ for $ A \subseteq X $. We consider a connected graph $b$ over $(X,m)$. Its Laplacian is given by
\begin{align*}
	\L f(x) = \frac{1}{m(x)} \sum_{y\in X} b(x,y)(f(x)-f(y))
\end{align*}
Furthermore, we allow for a potential $c\colon X \to \R$ to take negative values and consider the formal Schrödinger operator acting as $\mathcal{H} f = \L f + \frac{c}{m} f$
for $ f\in \F $. We assume that the Schrödinger operator is positive, i.e., there is a positive nontrivial function $v\in \F $ such that
\begin{align*}\mathcal{H}v \ge 0.
\end{align*}
By the Harnack inequality we have $v>0$, \cite[Lemma~4.5]{KPP20}.
In this case, by the Allegretto-Piepenbrink theorem, \cite[Theorem~4.2]{KPP20}, and \cite[Theorem 3.8]{Sch20}  there is a  self-adjoint operator $H$ on $\ell^2(X,m)$ such that $C_c(X)$ is form dense and which acts as $\mathcal{H}$ on  $D(H)$, see also \cite[Corollary~3.8]{KPP20} for the statement under weak restriction on $c$. This gives rise to a positivity preserving semigroup $e^{-tH},$ $t\ge 0 $, see \cite[Theorem~3.1, Corollary~3.5]{KPP20}. Furthermore, if $\H$ is subcritical, see e.g.\@ \cite[Theorem~5.12]{KPP20} for a characterization, there is  a Green function $ G $, see \cite[Definition~5.13]{KPP20}, and we define the space of potentials $\P$ accordingly.

We state the  version of the Liouville theorem for Schrödinger operators and give a sketch of the proof. 

\begin{theorem}[Liouville theorem -- Schrödinger operators]
		\label{thm:main_measure}
Let $H$ be a subcritical positive Schrödinger operator over $(X,m)$ and $0< v\in \F$ be such that $\mathcal{H}v \ge 0$.
Then, the following assertions are equivalent.
\begin{itemize}
\item[(i)] Every  harmonic function $h$ bounded by $|h| \le v$ is zero.
\item[(ii)] $G\mathcal{H}v = v$, i.e., $v\in \P$.
\item[(iii)] One has for all $t >0$ that
\begin{align*}
	e^{-tH}v+\int_0^t e^{-sH} \mathcal{H}v ds = v
\end{align*}
and
$$e^{-tH}v \to 0,\mbox{ as } t \to \infty\;\mbox{ pointwise}.$$
\item[(iv)] For all superharmonic $u \in \mathcal{P} $  
$$\sum_X   (\mathcal{H}u)vm = \sum_X u(\mathcal{H}v) m.$$
\end{itemize}	
\end{theorem}

\begin{remark}
	For certain graphs and potentials, the condition~(i) is known to hold. For instance, in \cite{BMP24}, the authors consider locally finite graphs over a discrete measure space $(G,\mu)$ and positive potentials $V$ that grow fast enough at infinity with respect to some proper, intrinsic metric of finite jump size. In \cite[Theorem~2.4]{BMP24}, it is shown that under a certain volume growth condition, all bounded solutions to $\mathcal{H}u=0$ must vanish for the killing term $c=V\mu$. The above Theorem~\ref{thm:main_measure} now applies with $v=1$, and we obtain validity of the assertions (ii)-(iv) as well for $\mathcal{H}v=\mathcal{H}1 = V$ and $\mu=m$. 
\end{remark}

The theorem can be also understood as a characterization of superharmonic potentials for positive Schrödinger operators. This perspective gives rise to the following immediate corollary.
\begin{corollary}	Let $H$ be a subcritical positive Schrödinger operator over $(X,m)$ and $v>0$ be a superharmonic potential. Then, every harmonic function $h$ bounded by $|h| \le v$ is zero.
\end{corollary}

 The proof is based on the ground state transform, see \cite[Section~4.2]{KPP20}, which allows to reduce the case of Schrödinger operators to the case of graphs with killing term and the counting measure. Denote the operator of multiplication with $u$ by $T_u$ from $C(X)$ to $C(X)$. 
 Then, there is a graph $(b_v,c_v)$ over $(X,v^2m)$ with $b_v(x,y) = b(x,y)v(x)v(y)$ and $c_v = mv \mathcal{H}v \geq 0$ such that the corresponding formal Laplacian $\mathcal{L}_{v,m}$ satisfies
\begin{align*}
	\mathcal{L}_{v,m} = 
	T_v^{-1} \mathcal{H} T_v
\end{align*}
on $\F_v=v^{-1}\F$. We observe that $T_v$ restricts to a unitary operator from $\ell^2(X,v^2m)$ to $\ell^2(X,m)$. Normalizing the measure, we obtain that 
 \begin{align*}
	\mathcal{L}_{v } = 
	T_{v^2m} \mathcal{L}_{v,m}
\end{align*}
is the formal Laplacian associated with the graph $(b_v,c_v)$ over the measure space $(X,1)$. 
We denote the Green function of $\mathcal{L}_v$ by $G_v$ and 	of $\mathcal{L}_{v,m}$ by $G_{v,m}$. Then, we have  following immediate lemma.
\begin{lemma}
\label{lem:ground_G} Let $H$ be a subcritical positive Schrödinger operator over $(X,m)$ and $0< v\in \F$ be such that $\mathcal{H}v \ge 0$. Then, for all $x \in X$,
\begin{align*}
v(x)^{-1}G \mathcal{H}v(x)=  G_vc_v(x)=G_{v,m}(c_v/(v^2m))(x)	.
\end{align*}
\end{lemma}
\begin{proof}
We observe that for a finite set $K \subseteq X$ 
the Dirichlet restriction $H_K$ of $H$ is injective, \cite[Lemma~5.15]{KPP20}, and
\begin{align*}
(H_{K_n})^{-1}	= T_v (L_{v,m,K_n})^{-1} T_v^{-1} = T_v (L_{v,K_n})^{-1} T_m T_v
\end{align*}
Hence, we have  by the finite approximation of the Green function, \cite[Theorem~5.16~(b)]{KPP20}, we have for all $x,y \in X$
\begin{multline*}
G(x,y)=\lim_{n\to\infty} (H_{K_n})^{-1}1_y(x) =\lim_{n\to\infty} v(L_{v,K_n})^{-1}(vm1_y)(x)\\
 ={m(y)v(y)v(x)}G_v(x,y).
\end{multline*}
and, therefore, for all $x \in X$,
\begin{multline*}
G\mathcal{H}v(x) =\sum_{y\in X} G(x,y) \mathcal{H}v(y)\\  =\sum_{y\in X} 
{m(y)v(y)}{v(x)}G_v(x,y)\frac{1}{v(y)m(y)}\Big(\mathcal{L}_{v}\frac{1}{v}\Big)v(y)   = {v(x)}G_vc_v(x).
\end{multline*}
This proof for $G_{v,m}$ is similar.
\end{proof}
 
Since the semigoup of $H$ is a posivitiy preserving operator on a discrete space, it admits a positive kernel. Hence, we can define $e^{-tH}f$ pointwise as a function from $X$ to $[0,\infty]$ for all $f\ge0$. The following lemma gives the connection between the semigroup of $H$ and the semigroup of $\mathcal{L}_{v,m}$.

\begin{lemma}
\label{lem:ground_sg} Let $H$ be a positive Schrödinger operator over $(X,m)$ and $0< v\in \F$ be such that $\mathcal{H}v \ge 0$. Then, for all $x \in X$,
\begin{align*}
e^{-tH}v(x)  =ve^{-tL_{v,m}}1(x)\quad\mbox{ and }\quad   e^{-tH} \mathcal{H}v(x)   =   e^{-tL_{v,m}}\frac{c_v}{vm}(x).
\end{align*}
\end{lemma}
\begin{proof}
	Since $T_v$ is a unitary operator from $\ell^2(X,v^2m)$ to $\ell^2(X,m)$, we have for all $t\ge 0$ that
\begin{align*}
e^{-tH} = T_v e^{-tL_{v,m}} T_v^{-1}.
\end{align*}
Since $v,\mathcal{H}v \ge 0$, we can approximate them monotonically by positive and compactly supported functions. We, therefore, can use  that the semigroups are positivity preserving to obtain the desired formulas.
\end{proof}
\begin{proof}[Proof of Theorem~\ref{thm:main_measure}]
(i) $\Leftrightarrow	$ (ii): Clearly, if $u$ is harmonic for $\mathcal{H}$, then $u/v$ is harmonic for $\mathcal{L}_{v}$.  Hence, the equivalence follows by Lemma~\ref{lem:ground_G} and Theorem~\ref{thm:main} for the counting measure. Then, by the ground state transform,  we have that every bounded harmonic function for $\mathcal{L}_{v,m}$ and consequently for $\mathcal{L}_{v}$ is zero. Hence, by Theorem~\ref{thm:main} for the counting measure, we have $G_vc_v=1$. Thus, by Lemma~\ref{lem:ground_G}, we have $G\mathcal{H}v=v$ as well. \medskip

(ii) $\Leftrightarrow	$ (iii): By Lemma~\ref{lem:ground_G} and following the proof of Lemma~\ref{lem:heat} by replacing the counting measure with the measure $m$, we have for all $t>0$ that
\begin{align*}
v^{-1}G\H v(x)=	G_{v,m}(c_v/m)(x) &=\lim_{t\to\infty}M_t(x) -\lim_{t\to\infty}e^{-tL_{v,m}}1(x),
\end{align*}
where $e^{-tL_{v,m}}1=\frac{1}{v}e^{-tH}v$ by Lemma~\ref{lem:ground_sg} and
\begin{multline*}
M_t(x) =e^{-tL_{v,m}}1(x)+\int_0^t e^{-sL_{v,m}}(c_v/(v^2m))(x) ds \\
 = 
\frac1{v}\left(e^{-tH}v(x)+\int_0^t e^{-sH} \mathcal{H}v(x) ds\right) 
\end{multline*}
by Lemma~\ref{lem:ground_sg} and $\mathcal{H}v=c_v/(vm)$. Hence,  the  equivalence holds by Theorem~\ref{thm:main}. \medskip

(ii) $\Leftrightarrow	$ (iv): This follows also immediately from Lemma~\ref{lem:ground_G}, $m\mathcal{H} u=v^{-1}\mathcal{L}_{v}(u/v )$, and the equivalence of (ii) and (iv) in Theorem~\ref{thm:main} for the counting measure, and $c_v = mv\mathcal{H}v
$.
\end{proof}

\section{Examples and applications}\label{sec:examples}

\subsection{The standard lattice $\Z^d$}
We first consider the case of the  lattice $\Z^d$ with the standard Laplacian and the counting measure which is given as
\begin{align*}\Delta f(x) = \sum_{ |y-x|=1} (f(x)-f(y)).
\end{align*}	
The corresponding graph $b$ is given by $b(x,y)=1$ if $|x-y|=1$ and $b(x,y)=0$ otherwise and the killing term is given by $c=0$. It is well-known that the graph $b$ is recurrent for $d=1,2$ and transient for $d\ge 3$. Moreover, it is also well-known that every bounded harmonic function on $\Z^d$ is constant.

For a subset $U\subseteq\Z^d$, we denote by $\Delta_U$ the restriction of $\Delta$ to $U$ with Dirichlet boundary condition, which corresponds to the graph $(b_U,d_U)$, and we write $G^U$ for its Green operator.

We draw the following immediate consequence of Section~\ref{sec:Rec}.
\begin{corollary}\label{cor:Gc_increases_with_d} 
	Let $\emptyset\neq U\subsetneq  \Z^d $.
	\begin{itemize}
		\item [(a)] If $d=1,2$, then every bounded harmonic function of $\Delta_U$ on $U$ is zero.
		\item [(b)] If $d\ge 3$ and $X\setminus U$ is finite, then there are non-zero bounded harmonic functions of $\Delta_U$ on $U$.
	\end{itemize}
\end{corollary}
\begin{proof}

For $d=1,2$  the graph $b_U$ over $U$ is recurrent as a subgraph of a recurrent graph, and, therefore, $G^U d_U(0)=1$ by Theorem~\ref{thm:recurrence}. Note that we did not assume $U$ to be connected but we can obviously argue on every connected component of $U$ separately.
	
For $d\ge 3$ the graph $b_U$ over a co-finite $U$  is transient (since if $1\notin\D_0$, then it easily follows that $1_U\notin\D_{0,U}$) and finiteness of $X\setminus U$ ensures that $d_U\in\ell^1(X)$. Hence, $G^U d_U<1$ by Theorem~\ref{thm:recurrence}.

The statement then follows in either case by Theorem~\ref{thm:main}.
\end{proof}

\begin{remark}
 The proof of (a) works on any recurrent graph. Hence, if $(b,0)$ is recurrent and $\emptyset \neq U \subsetneq X$, then every bounded harmonic function for the Dirichlet Laplace on $U$ (the graph $(b_U,d_U)$) is zero.
\end{remark}

\subsection{Subsets of half-spaces of the standard lattice}
We next consider subsets  of half-spaces of the standard lattice $\Z^d$ which includes half-spaces and orthonants. 
A half-space of $\Z^d$ is given by
$\{x\in \Z^d\mid x_j\ge0\}$ for some $j\in \{1,\ldots,d\}$.

\begin{theorem}\label{thm:Hmn}
Let $ U $ be a connected subset in $\Z^d$ which is included in a half-space. Then, every bounded harmonic function of $\Delta_U$ on $U$ is zero.
\end{theorem}
In \cite{BMS,Raschel} a similar result was proven for positive harmonic functions. The result above can also be deduced from the Phragmen-Lindelöf type theorems proven in \cite{Guadie} but as the proof is simple and illustrates our main result we give it here.

\begin{proof}[Proof of Theorem~\ref{thm:Hmn}]
Let $Y_n$  be  the simple random walk on $ \Z^d $
and let $ Y_n^U $ be the simple random walk on $ U $ which is killed at the boundary of $U$. Let $H=\{x\in \Z^d \mid x_j\ge 0\}$ be the halfspace including $U$. Then, we have for the exit times
\begin{align*}
	\tau_U&:=\inf\{n\ge 0\mid Y_n^U\notin U\}\\
	\tau_H&:= 
	\inf\{n\ge 0\mid Y_n\notin H\}.
\end{align*}
Clearly,
$$\mathbb{P}_x(\tau_U<\infty)\ge \mathbb{P}_x(\tau_H<\infty)$$ for all $x\in U$.
Furthermore, the projection of the random walk $Y_n$ on the $j$-th coordinate is a one-dimensional simple random walk which is recurrent. Hence,  $\mathbb{P}_x(\tau_U<\infty)\ge\mathbb{P}_x(\tau_H<\infty)=1$ for all $x\in U$.  Thus, the theorem follows from Theorem~\ref{thm:Green}~(ii.c) and Theorem~\ref{thm:main}.
\end{proof}

\begin{remark}
In view of Proposition~\ref{proposition:characterization thickness} this implies that the complement of any set contained in a half-space is thick.
\end{remark}

\subsection{Percolation clusters in the supercritical regime}

Next, we consider the case of percolation clusters in the supercritical regime. We consider the standard lattice $\Z^d$ with $d\ge 2$ and the standard weights and the counting measure. We perform Bernoulli bond percolation on $\Z^d$ with parameter $p\in (0,1)$, i.e., we keep each edge independently with probability $p$ and delete it with probability $1-p$. We can model the arising random graph as a translation dynamical system 
$\big( \{0,1\}^{E_d}, \mathbb{P}_p, (\sigma_n) \big)$, where $E_d$ denotes the set of edges in $\Z^d$, and $\mathbb{P}_p$ is the product measure obtained from the Bernoulli distribution for every edge. Clearly, we can identify each $\omega \in \{0,1\}^{E_d}$ with a subgraph of the lattice graph, with the edge $(x,y) \in E_d$ appearing in this subgraph if and only if $\omega((x,y)) = 1$. 
Further, $\sigma = (\sigma_n)_{n \in \Z^d}$ denotes the ergodic, in fact mixing, and essentially free measure preserving action of $\Z^d$ on $\{0,1\}^{E_d}$ by translations such that $\sigma^n \omega((x,y)) = \omega((x+n,y+n))$ for $n \in \Z^d$ and $(x,y) \in E_d$. 
It is well-known that there is a critical parameter $p_c(d)\in (0,1)$ such that for all $p>p_c(d)$, for  $\mathbb{P}_p$-almost every $\omega$, there is a unique infinite cluster $\mathcal{C}_\infty(\omega)$ and for all $p<p_c(d)$ there are almost surely only finite clusters. 

Assuming now $p > p_c(d)$, we consider the operator with Dirichlet boundary conditions. Let $b_\omega$ be the graph of  the unique infinite cluster $\mathcal{C}_\infty(\omega)$. Furthermore, for the Laplacian $\Delta$ on $\Z^d$, we let $c_\omega\colon 	\mathcal{C}_\infty(\omega) \to \{0,\ldots,2d-1\}$ be given by
$$c_\omega(x)=\#\{ (x,y) \in E_d \mid \omega((x,y)) = 0 \},$$
i.e., $c_\omega$ counts the number of adjacent edges to $x$ in $\Z^d$ that have been deleted. We denote the corresponding Laplacian with Dirichlet boundary condition by $$\Delta_\omega=\mathcal{L}_{b_\omega,c_\omega}$$ and we say $u$ is a harmonic function on $\mathcal{C}_\infty(\omega)$ if $u$ vanishes outside $\mathcal{C}_\infty(\omega)$ and $\Delta_\omega u=0$ on $\mathcal{C}_\infty(\omega)$.

\medskip

In \cite[Theorem~4]{Barlow} the author shows the a.s.\@ absence of non-constant positive harmonic functions for the Laplacian $\mathcal L_{b_\omega,0}$, which can be interpreted as the Laplacian on $\mathcal{C}_\infty(\omega)$ with Neumann boundary condition. An analogous statement for bounded harmonic functions can be deduced from the results in \cite{Kaimanovich90}, see also  \cite[Lemma~4.6]{BLS99}. Our main result on the Laplacians $\Delta_{\omega}$ with Dirichlet boundary condition is the following.

\begin{theorem} \label{thm:percolation}
	Let $d\ge 2$ and $p>p_c(d)$. Then, for $\mathbb{P}_p$-almost every $\omega$, every bounded harmonic function on $\mathcal{C}_\infty(\omega)$ is zero.
\end{theorem}

\begin{proof}
	For $d=2$, the simple random walk on $\Z^2$ is recurrent and together  with the observation that $c_{\omega} \neq 0$ almost surely, 
	 the statement follows in this case from Theorem~\ref{thm:recurrence}~(a).
	So we assume now that $d \geq 3$. We will show that for almost all $\omega$ and $o \in \mathcal{C}_{\infty}(\omega)$, the simple random walk starting at $o$ will leave $\mathcal{C}_{\infty}(\omega)$ with probability one. This will imply that one will $\mathbb{P}_p$-almost surely pass along a deleted edge. 
	Let $o \in \Z^d$ and set $F=\{ \pm e_j \mid 1 \leq j \leq d \}$, where $e_j$ denotes the standard unit vector which is $1$ at coordinate $j$. We denote the simple random walk on $\Z^d$ starting at $o$ by $(Y_k)$, where $Y_k$ is modeled over the product space $F^{\N}$, endowed with product measure $\mathbb{P}_s$. For $f \in F^{\N}$, we denote its $k$-th coordinate by $f_k$. Defining by $\theta((f_k)) = (f_{k+1})$ the shift transformation on $F^{\N}$, we obtain a skew product transformation $T$ on the probability space $(F^{\N} \times \{0,1\}^{E_d}, \mathbb{P}_s \times \mathbb{P}_p)$, namely 
	\[
	T(f,\omega) = \big(  \theta f,\, \sigma^{-f_1} \omega \big).
	\]
	Note that by transience of the simple random walk on $\Z^d$, $d \geq 3$, the realizations of $\mathbb{P}_s$-almost all $f$ as paths in $\mathbb{Z}^d$ starting at $o$ eventually leave every finite set. Combining this with the fact that $\theta$ is ergodic and $\sigma$ is mixing, it follows from standard arguments from ergodic theory (using e.g.\@  characterizations of ergodicity and mixing as in \cite[Theorem~8.10 and Theorem~9.4]{EFHN}) that $T$ must be ergodic. We define the set $C_o = \{ \omega \mid o \in \mathcal{C}_{\infty}(\omega)  \}$ and denote its complement by $C_o^c$. By Birkhoff's ergodic theorem, for $\mathbb{P}_s \times \mathbb{P}_p$-almost all $(f,\omega)$ we obtain 
	\begin{align*}
		\lim_{N \to \infty} \frac{1}{N} \sum_{k=0}^{N-1}  1_{F^{\N} \times C_o^c}\big( T^k(f,\omega)\big) = \mathbb{P}_p(C_o^c) = 1- \mathbb{P}_p(C_o).
	\end{align*}
	The right hand side in the above equation is positive since $0 < \mathbb{P}_p(C_o)<1$. This shows that for almost all $(f, \omega)$, there is
	 one (in fact, there are infinitely many) $k \in \N$, such that for $g_k = \sum_{i=1}^k f_i$,
	\[
	T^k(f,\omega) = \big( \theta^k, \sigma^{-g_k}\omega \big) \in F^{\N} \times C_o^c.
	\] 
	Inspection of the second component gives $o \notin  \mathcal{C}_{\infty}(\sigma^{-g_k}\omega)$, and equivalently, $o+g_k \notin \mathcal{C}_{\infty}(\omega)$. We have thus shown that with probability one, $Y_k$ leaves $\mathcal{C}_{\infty}(\omega)$ for $\mathbb{P}_p$-almost all $\omega$. 	
	The validity of the theorem now follows from   Theorem~\ref{thm:Green}~(ii.c) in combination with Theorem~\ref{thm:main}.
\end{proof}

\subsection{Weakly spherically symmetric graphs} 

In this final section we give a short discussion of weakly spherically symmetric graphs, see e.g. \cite{KLW13} or \cite[Chapter~9]{KLW}. By the  reasoning in the introduction, we can interpret them as subgraphs of a larger graph where the killing term constitutes the boundary of the subgraph. For weakly spherically symmetric graphs with measure and without killing term, the $\ell^1$-Liouville property for positive summable functions has been characterized in \cite[Theorem~3.5]{AS23} by stochastic completeness of the graph. In the theorem below, we provide a characterization of non-existence of non-trivial bounded harmonic functions in the presence of a killing term.
A graph $(b,c)$ over $X$ is called {\em weakly spherically symmetric} with respect to a vertex $o\in X$ if for all $x\in X$ and $r=|x|=d(x,o)$, the functions
\begin{align*}b_\pm(x) &=  \sum_{y\in S_{r\pm1}(o)} b(x,y)
\end{align*}
and $c$ depend only on $r$, where  $S_r=S_r(o)=\{x\in X\mid d(x,o)=r\}$ denotes the sphere of radius $r$ around $o$ with respect to the path distance function. We also write $B_r=B_r(o)$ for the corresponding closed balls. Further, we extend $b_{\pm}$ and $c$ in a natural way to mappings from finite subsets of $X$ by taking their sum over these sets.

\begin{theorem}
	Let $(	b,c)$ be a weakly spherically symmetric graph over $X$ with respect to $o\in X$. Then, every bounded harmonic function is zero if and only if
	\begin{align*}
		\sum_{r=0}^\infty\frac{ c(B_r )}{b_+(S_r )}=\infty.
	\end{align*}
	Moreover, one has $Gc=1$ if and only if $G^bc=\infty$ for the Green function $G^b$ of the graph $b$.
\end{theorem}
\begin{proof}
	The Green function $G^b$ of the weakly spherically symmetric graph $b$ can be computed explicitly, see e.g. \cite[Section~9.3]{KLW}, and one finds that for all $x\in X$ with $r=|x|$,
	\begin{align*}G^b(x,o) = \sum_{k=r}^\infty \frac{1}{b_+(S_k )}.
	\end{align*}
	Hence, we have for all $x\in X$ with $r=|x|$,
	\begin{align*}G^b c(o) = \sum_{r=0}^\infty \sum_{x\in S_r(o)} c(x) G^b(o,x) = 
		\sum_{r=0}^\infty\sum_{k=r}^\infty \frac{ c(S_r )}{b_+(S_k )}=
		\sum_{k=0}^\infty \frac{c(B_k )}{b_+(S_k )},
	\end{align*}
	by Fubini's theorem. On the other hand, there is a unique spherically symmetric harmonic function $u$ with $u(o)=1$ for $\L=\L_{b,c}$ which is  recursively given by
	\begin{align*} 
		 u(r+1)-u(r)=\frac{1}{b_+(S_r )} \sum_{k=0}^{r} c(S_k )u(k),
	\end{align*}
	 by \cite[Lemma~9.16]{KLW}. Clearly, $u$ is increasing as $c\ge0$, and, therefore, $u(k)\ge u(0)= 1$ for all $k\ge 0$. So, summing the above formula again telescopically from $0$ to $r-1$, we find the following explicit formula for  $r\ge 1$ and estimate 
\begin{align*}u(r) &= 1 + \sum_{n=0}^{r-1}\frac{1}{b_+(S_n )}
	\sum_{k=0}^{n}c(S_k ) u(k)  \ge 1 + \sum_{n=0}^{r-1}\frac{c(B_n )}{b_+(S_n )}
	  .
\end{align*}
Hence, if $G^bc(o)=\infty$, then $u(r)\to\infty$ as $r\to\infty$ and, therefore, the unique spherically symmetric harmonic function is unbounded. Therefore, by Theorem~\ref{thm:Liouville} (applied to the corresponding weighted line graph) every spherically symmetric bounded positive subharmonic function is zero. If there was an arbitrary bounded positive subharmonic function, then its average is a  spherically symmetric bounded positive subharmonic function and, hence, zero.
Thus, there is no non-zero bounded harmonic function by Theorem~\ref{thm:Liouville}. 

Conversely, if $G^bc(o)<\infty$, then $Gc(o)<1$ by Lemma~\ref{lem:recurrence2}.

Thus, the statement and also the ``in particular'' statement follow by Theorem~\ref{thm:main}.
\end{proof}

\textbf{Acknowledgement.} The authors acknowledge the financial support of the DFG and the hospitality of the IIAS Jerusalem.

\printbibliography


\end{document}